\documentclass[11pt]{amsart}

\usepackage{amsmath,amssymb,amsthm}
\usepackage{amsmath,amssymb,amsthm,amscd}
\usepackage[frame,cmtip,arrow,matrix,line,graph,curve]{xy}
\usepackage{graphpap, color}
\usepackage[mathscr]{eucal}
\usepackage{color}
\numberwithin{equation}{section}

\usepackage{mathrsfs}
\usepackage{graphicx}
\usepackage{pstricks}

\title[Balanced metrics on Calabi-Yau threefolds]{Balanced metrics on
non-K\"ahler Calabi-Yau threefolds}

\author{Jixiang Fu}
\author{Jun Li}
\author{Shing-Tung Yau}

\address{Institute of Mathematics\\ Fudan University \\ Shanghai 200433\\
China}
 \email{majxfu@fudan.edu.cn}
\address{Department of Mathematics, Stanford University\\ Stanford, CA94305\\ USA}
\email{jli@math.stanford.edu}
\address{Department of Mathematics\\ Harvard University\\ Cambridge,
MA 02138\\ USA} \email{yau@math.harvard.edu}
\date{}

\newtheorem{prop}{Proposition}[section]
\newtheorem{theo}[prop]{Theorem}
\newtheorem{lemm}[prop]{Lemma}
\newtheorem{coro}[prop]{Corollary}
\newtheorem{rema}[prop]{Remark}

\newtheorem{defi}[prop]{Definition}

\newtheorem{prob}[prop]{Question}

\DeclareMathOperator{\Aut}{Aut}

\def\neib{{\mathcal N}}
\def\cO{{\mathcal O}}

\def\cX{{\mathcal X}}

\def\bE{{\mathbf E}}

\def\bt{{\mathbf t}}

\def\bs{{\mathbf s}}

\def\bU{{\mathbf U}}

\def\brr{{\mathbf r}}

\def\ZZ{{\mathbb Z}}

\def\RR{{\mathbb R}}
\def\CC{{\mathbb C}}

\def\and{\quad{\rm and}\quad}

\def\bl{\bigl(}
\def\br{\bigr)}

\def\bs{{\mathbf s}}

\def\bt{{\mathbf t}}

\def\dbar{\bar\partial}

\def\defeq{\triangleq}

\def\eps{\epsilon}

\let\lra=\longrightarrow

\def\pri{^{\prime}}

\def\Po{{\mathbf P}^1}

\def\sub{\subset}
\def\sta{^{\ast}}

\def\upmo{^{-1}}

\def\sm{\mathrm{sm}}

\def\norml{\parallel\!}
\def\normr{\!\parallel}

\def\lab#1{\label{#1}[{#1}]\  }

\def\vsp{\vskip3pt}
\def\sumt{^{\oplus 2}}

\def\dbar{\bar\partial}

\def\neib{U}
\def\bbr{\mathbf{r}}
\def\ba{{\mathbf a}}
\def\bb{{\mathbf b}}

\def\va{\vskip4pt}
\def\ha{\frac{1}{2}}

\def\lab{\label}

\begin{document}

\maketitle

\begin{abstract}
We construct balanced metrics on the family of non-K\"ahler
Calabi-Yau threefolds that are obtained by smoothing after
contracting $(-1,-1)$-rational curves on  a  K\"ahler
Calabi-Yau threefold. \black As an application, we construct
balanced metrics on complex manifolds diffeomorphic to  the
connected
  sum \black of $k\geq 2$ copies of $S^3\times S^3$.
\end{abstract}

\section{Introduction}

We construct balanced metrics on the class of complex threefolds
that are  obtained by \black conifold transitions of K\"ahler
Calabi-Yau threefolds; this class includes complex structures on
 the \black connected sum of $k\ge 2$ copies of $S^3\times
S^3$.

A central problem in studying  compact  complex manifolds is to find
special hermitian metrics on them. (All complex manifolds in this
paper are compact, unless otherwise stated.) The most distinguished
class of metrics on complex manifolds are K\"ahler metrics. A
K\"ahler metric is a hermitian metric whose hermitian form $\omega$
satisfies $d\omega=0$. K\"ahler metrics offer many advantages: their
 (hermitian) \black connections are torsionless; their $d$,
$\partial$ and $\dbar$-harmonic forms coincide, which lead to the
Hodge structure on their cohomology groups. The drawback is that
many important complex manifolds do not admit K\"ahler metrics.

In search for a wider class of special metrics on an $n$-dimensional
complex manifold $X$, since the vanishing
$d\omega^k=0$ automatically yields
$d\omega=0$ when $2\le k\le n-2$ (see
\cite{GH}), the only weaker condition along this line is the balanced
condition
$$d\omega^{n-1}=0.
$$
(One can generalize the K\"ahler condition along other directions,
like the pluriclosed metric: $\partial\dbar\omega=0$. In this paper,
we concentrate on  balanced metrics.)

The balanced metrics on $X$ (an $n$-dimensional complex manifold)
form an important class of hermitian metrics. First, the form
$\omega^{n-1}$ defines a cohomology class in $H^{2n-2}(X,\RR)$, thus
can be used to define the degree of vector bundles on $X$; the
balanced metrics also occur as part of the Strominger system, a
system that generalizes the complex Monge-Ampere equations and
hermitian-Yang-Mills equations (see \cite{Str}). Paired with
cohomologically \black K\"ahler requirement on the manifold $X$,
(i.e. the validity of the $\partial\dbar$-Lemma on $X$), we expect
 that a \black  balanced metric would yield properties
resembling that of a K\"ahler metric.

The existence of balanced metrics is also more robust than that of
K\"ahler metrics; more so when the base manifold is cohomologically
K\"ahler. For a pair of birational complex manifolds, Alessandrini
and Bassanelli \cite{AB2,AB3} proved that one admits balanced
metrics if the other admits balanced metrics; when $X$ is
cohomologically K\"ahler and has balanced metrics, then small
deformations of the complex structure of $X$ is also cohomologically
K\"ahler \cite{Voi,Wu} and admits balanced metrics \cite{Wu}.

This leads to the natural question whether balanced metrics are
preserved under singular transitions of the underlying manifold. A
singular transition of a complex manifold $Y$ is a contraction $Y\to
X_0$ followed by a smoothing $X_0 \rightsquigarrow X_t$, (i.e. $X_t$
are small deformations of $X_0$ such that $X_t$ are smooth for
general $t$.) The simplest such case is the conifold transition:
%

\begin{defi}\label{coni-def}
A conifold transition consists of a smooth compact threefold $Y$, a
holomorphic map to a singular complex space $\pi: Y\to X_0$ and an
analytic family of complex spaces $X_t$, $t\in\Delta\sub \CC$, such
that
\begin{enumerate}
\item
$X_0$ is compact and smooth away from a finite set $\Lambda=\{p_1,\cdots,p_\ell\}$;
\item $\pi\upmo(p_i)\defeq E_i$ are $(-1,-1)$-curves; i.e., they are smooth rational curves, and
the normal bundles $N_{E_i/Y}$ are isomorphic to $\cO_{E_i}(-1)^{\oplus 2}$;
\item
$\pi|_{Y-\pi\upmo(\Lambda)}: Y-\pi\upmo(\Lambda)\to X_0-\Lambda$
is a biholomorphism;
\item $X_t$ are compact smooth complex manifolds for $t\ne 0$.
\end{enumerate}
\end{defi}

In this paper, we prove the existence of balanced metrics under
conifold transitions.

\begin{theo}
Let $Y$ be a smooth K\"ahler Calabi-Yau threefold and let $Y\to
X_0\rightsquigarrow X_t$ be a conifold transition.
Then for sufficiently small $t$, $X_t$ admits smooth balanced metrics.
\end{theo}

Here a smooth Calabi-Yau threefold is a three dimensional complex
manifold with finite fundamental group and trivial canonical line
bundle. There are plenty of conifold transitions of K\"ahler
Calabi-Yau threefolds. Given such a threefold $Y$,   let $\bf
E$  be a union of mutually disjoint $(-1,-1)$-curves $E_i$. \black
By contracting  $\bf E$, \black we obtain a singular complex
space $X_0$. When the homology classes $[E_i]\in H_2(Y,\ZZ)$
satisfies the criterion of Friedman \cite{Fri1,Fri2}, $X_0$ can be
smoothed to a family of Calabi-Yau threefolds $X_t$. The theorem
states that for sufficiently small $t$, all $X_t$ have balanced
metrics.

The connected sum $\#_k (S^3\times S^3)$ of $k\ge
2$ copies of $S^3\times S^3$ can be given a complex structure in
this way \cite{Fri2,LT}. As a corollary of the Theorem,

\begin{coro}\label{cor1.3}
The complex structures on $\#_k (S^3\times S^3)$ for any $k\geq 2$
constructed from the conifold transitions admit balanced metrics.
\end{coro}

On the other hand, according to Lemma 2 in \cite{Bo1}, any
pluriclosed metric $\omega$ on $\#_k(S^3\times S^3)$ can be written
as $\omega=\partial\bar\phi+\bar\partial\phi$ for a $(1,0)$-form
$\phi$. We claim that in this case there is no balanced metric on
it. Otherwise, a balanced metric $\tilde\omega$ would give
$0<\int_{\#_k(S^3\times S^3)}\omega\wedge\tilde\omega^{n-1}=0$, a
contradiction.

Combining Corollary \ref{cor1.3} and the above discussion, we prove
a result stated in \cite{Bo1}. (The proof of this statement in
\cite{Bo1} was incomplete; the reason given in \cite{Bo1} for $T=0$
is insufficient.)

\begin{coro}
There exists no pluriclosed metric on the complex structures on $\#_k(S^3\times S^3)$ for any $k\geq 2$ constructed from the conifold transition.
\end{coro}

This shows that the balanced metrics are the {\sl only} known special  hermitian metrics
on these manifolds (c.f. \cite{Bo1}). We add that in \cite{Bo2} it is proved that their holomorphic tangent bundles are stable with respect to any Gaudchon metric.
\vsp


We believe that the theorem will play an important role in
investigating the geometry of Calabi-Yau threefolds within the
framework of Reid's conjecture. To shed lights on the immense
collection of diverse Calabi-Yau threefolds, Reid conjectured that
all Calabi-Yau threefolds are connected by deformations and singular
transitions \cite{Reid}.  The current work is a step to study
Calabi-Yau threefolds in the framework of metric geometry along
Reid's conjecture.

Our proof of the Theorem is partially constructive in that we
construct balanced metrics $g_t$ on $X_t$ with prescribed limiting
behavior near the singularities of $X_0$. This helps to investigate
the solutions to the Strominger system of supersymmetry with torsion
under  the  conifold transition. Recall that the Strominger system
is an elliptic system on a pair $(g,h)$ of a hermitian metric $g$ on
a Calabi-Yau threefold $Y$ and a hermitian metric $h$ on a vector
bundle $V$ on $Y$ (c.f. \cite{Str, LY1, FY, BBFTY,
FTY}). This system includes an equation on the hermitian form
$\omega$ of $g$:
\begin{equation*}\label{104} d\sta
\omega=\sqrt{-1}(\dbar-\partial)\ln\|\Omega\|_\omega,
\end{equation*}
which is equivalent to the balanced condition \cite{LY1}:
$$d(\parallel\!\Omega\!\parallel_\omega\omega^2)=0.
$$
(Here $\Omega$ is a holomorphic $3$-form of the Calabi-Yau
threefold.) We hope that the solutions to the  Strominger system for
$Y$ can be prolonged through conifold transitions. One can also consult the discussion on this point from CFT in \cite{Ad}.

We add that there are explicit existence results on balanced
metrics. Goldstein and Prokushkin \cite{GP} constructed balanced
metrics on torus bundles over $K3$ surfaces and {over} complex
abelian surfaces (cf. \cite{DRS} and \cite{BBDG}). Later, D.
Grantcharov, G. Grantcharov and Poon \cite{GGP} constructed CYT
structures on torus bundles over more general compact K\"ahler
surfaces; as a consequence they constructed CYT structures on
complex manifolds of topological type $(k-1)(S^2\times S^4)\#
k(S^3\times S^3)$ for $k\geq 1$. However, the canonical line
bundles of these complex manifolds are non-trivial.  Note that for
compact complex manifolds with trivial canonical line bundles, the
existence of CYT structures is equivalent to the existence of
balanced metrics \cite{LYZ}. Along this line, our construction
provides  CYT structures on a large \black class of threefolds,
including those of types $\#_{k\ge 2}( S^3\times S^3)$. \vsp

We now outline the proof of the theorem. Our first step is to modify
a  K\"ahler metric on $Y$ near the contracted curves $E_i$ to get a
balanced metric $\omega_0$ on the contraction $X_0$ so that near the
singularities of $X_0$ the metric $\omega_0$ coincides with the K\"ahler
Ricci-flat metric of Candelas-de la Ossa's (see \cite{CO}).

After this, we deform $\omega_0$ to
smooth almost balanced hermitian metrics $\omega_t$ on $X_t$ so that
they are K\"ahler and Ricci-flat near the singular points of $X_0$.
We achieve the Ricci-flatness by using the deformation of
Candelas-de la Ossa's metric on the cone singularity to smooth
Ricci-flat metrics on the smoothing of the cone singularity.

To get balanced metrics, we consider perturbation
$\omega_t^2+\theta_t+\bar\theta_t$,
with $\theta_t=i\partial \mu_t$ for $\mu_t$ a $(1,2)$-form on
$X_t$ solving
\begin{equation*}
i\partial_t\dbar_t\mu_t=\dbar_t\omega_t^2 \quad\text{subject
to}\quad  \mu_t\perp_{\omega_t}\ker
\partial_t\dbar_t.
\end{equation*}
This way, $d(\omega_t^2+\theta_t+\bar\theta_t)=0$ automatically. We
then prove that the $C^0$-norms
$\parallel\!\!\theta_t\!\!\parallel_{C^0,\omega_t}\to 0$ as $t\to
0$. Thus $\omega_t^2+\theta_t+\bar\theta_t$ is positive definite for
small $t$; $(\tilde\omega_t)^2=\omega_t^2+\theta_t+\bar\theta_t$ is
solvable, and $\tilde\omega_t$ is a family of balanced metrics on
$X_t$.

The technical part is to control  the norms
$\parallel\!\!\theta_t\!\!\parallel_{C^0,\omega_t}$. To this end, we
choose $\gamma_t$ to be the solution to the Kodaira-Spencer equation \cite{KS}
$E_t(\gamma_t)=\dbar\omega_t^2$ subject to $\gamma_t\perp_{\omega_t}
\ker E_t$. The solution $\gamma_t$ satisfies $\partial_t\gamma_t=0$
and $\mu_t=-i\dbar_t^\ast\partial_t^\ast\gamma_t$. Applying the
elliptic estimates, the $L^2$-estimates and the vanishing
theorem \black of $L^2$-cohomology groups, we prove that
$$\lim_{t\to0}|t|^{\kappa}\cdot\parallel\!\!\theta_t\!\!\parallel^2_{C^0,\omega_t}=0\quad\text{for}
\quad \kappa>-\frac 4 3;
$$
this is more than enough to get the desired bound on $
\parallel\!\!\theta_t\!\!\parallel_{C^0,\omega_t}$.
Section 3 and 4 are devoted to
prove this estimate. \vsp

The above construction of the family of hermitian metrics
$\omega_t$ and the estimate on the perturbation terms $\theta_t$
provide a precise control on the local behavior of the metrics
$\tilde\omega_t$ near the singularities of $X_0$. Such information
will be useful in the further study of the geometry of $X_t$. For
instance, using this M.-T. Chuan \cite{Chuan} has proved certian existence of
Hermitian-Yang-Mills metrics on bundles over $X_t$.

It is worthwhile to compare this approach with a possible approach using
Michelsohn's existence criterion of balanced metrics \cite{Mic}. Let
$Y\to X_0\rightsquigarrow X_t$ be a conifold  transition of the
\black Calabi-Yau threefold and suppose $Y$ is cohomologically
K\"ahler and has balanced metrics. In case $X_{t}$ does not have
balanced metrics, by Michelsohn's criterion we find  a non-zero
positive $(1,1)$-current $T_t$ on $X_t$ of the form
$T_t=\bar\partial S_{t}+\partial \bar S_{t}$ with $(1,0)$-current
$S_t$. Suppose  $X_{t_k}$ has no balanced metrics for a sequence
$t_k\to 0$,  then after normalization and passing to a subsequence,
we find a non-zero positive $\partial\bar\partial$-closed
$(1,1)$-current $T_0$ on $Y-E$ that is a weak limit of the $T_{t_k}$
mentioned.  If we can show that $T_0$ extends to a non-zero positive
current $\tilde T_0$ on $Y$ such that $\tilde T_0=\partial \bar
S+\bar\partial S$ for a $(1,0)$-current $S$ on $Y$, we obtain a
contradiction by applying Michelsohn's criterion to our assumption
that $Y$ has balanced metrics.

The extension is guaranteed if {\sl $T_0(\Phi)=0$ for any $d$-closed $(2,2)$-form $\Phi$ on
$Y-E$ with compact support.} One possible approach to such proof is to establish estimates
on a family of $(2,2)$-forms $\theta_t'$ (similar to the $\theta_t$ mentioned before) on $X_t$:
$$\parallel\!\!\theta_t'\!\!\parallel_{C^0,\omega_t'}\to 0\quad\text{as}\quad t\to 0.
$$
Here $\omega'_t$ are the hermitian metrics on $X_t$ that are the
restriction to $X_t$ of a smooth hermitian metric on
$\cX=\coprod_{t}X_t$. (Note for conifold transitions, $\cX$ is a
smooth, non-compact four-fold). At the moment we are unable to prove
this estimate. Though this is weaker than the estimate
$t^{\kappa}\cdot\parallel\!\!\theta_t\!\!\parallel^2_{C^0,\omega_t}\to
0$ mentioned earlier, we can prove the stronger estimate because we
use essentially   the Ricci-flatness \black of $\omega_t$ near
the singularities of $X_0$.

We hope that a refined version of this suggested approach will be useful to attack the
question on balanced metrics via singular transitions.

\begin{prob}
Let $Y$ be a compact cohomologically K\"ahler complex manifold and let
$Y\to X_0\rightsquigarrow X_t$ be a singular transition. Suppose $Y$
has a balanced metric. Does $X_t$ admit balanced metrics for
sufficiently small $t$?
\end{prob}

It will be too optimistic to believe that this question has an
affirmative answer in general. The case of threefolds (or Calabi-Yau
threefolds) holds more hope. Our theorem is the first step toward
answering this question. A more detailed understanding of this
question will be important to the metric geometry of threefolds.

\vsp

{\bf Acknowledgement.} The authors would like to thank  P.-F. Guan,
J.-X. Hong, Q.-C. Ji, L. Saper, V. Tossati and Y.-L. Xin for useful
discussions. The first named author is supported partially by NSFC
grants 10771037 and 10831008; the second named author is supported
partially by NSF grant NSF-0601002, and third named author is
supported partially by NSF grants DMS-0306600 and PHY-0714648.

\section{Balanced metrics with conifold singularity}
\def\co{\mathrm{co}}


Let $(Y,\omega)$ be a K\"ahler threefold. Let $E\sub Y$ be a
$(-1,-1)$-curve. By contracting $E$ we obtain a variety $X_0$ with
ordinary double point singularity. In this section, by modifying the
$4$-form $\omega^2$ we construct a balanced metric $\omega_0^2$ on
$Y- E$ that coincides with Candelas-de la Ossa's cone Ricci-flat
metric near the singular point of $X_0$.

We begin with setting up the convention for the geometry of $Y$ near the $(-1,-1)$-curve $E$.
We let $L$ be the degree $-1$ line bundle on
$E$; we pick a
neighborhood $U$ of $E$ in $Y$ that is biholomorphic to a disk
bundle in $L\sumt$.

To give coordinates to $U$, we fix an isomorphism $E\cong \Po$, pick an $\infty\in E$ and
let $z\in E-\infty=\CC$ be the standard coordinates of $\CC$. Using
$L\sumt|_{E-\infty}\equiv \CC_{E-\infty}\sumt$, and taking $e_1$ and
$e_2$ be the standard basis of $\CC_{E-\infty}\sumt$,  we give
$L\sumt|_{E-\infty}$ the coordinates $(z,u,v)$, meaning the point
$ue_1+ve_2$ over $z\in E-\infty$.

We let $\bbr\ge 0$ be the function
\begin{equation}\lab{bbr}
\bbr(z,u,v)^2=(1+\mid z\mid ^2)(\mid u\mid ^2+\mid v\mid
^2).
\end{equation}
A direct check shows that this function
extends to a smooth hermitian metric of $L\sumt$. Using $\bbr$, we agree that $U\sub Y$ (containing $E$)
is biholomorphic to
the open unit disk in $L\sumt$. For $1\geq c>0$, we let
\begin{equation}\lab{U}
\neib(c)=\{(z,u,v)\in\neib\mid  \bbr(z,u,v)< c\}\sub U(1)=U.
\end{equation}
As $U\sub Y$ is viewed as an open neighborhood of $E\sub Y$, using the above inclusion,
$U(c)$ for $0<c<1$ are open neighborhoods of $E\sub Y$ as well.

\vsp

We recall Candelas-de la Ossa's metric on $\neib$. To make the
forthcoming manipulation more tractable, since both $L\sumt$ and
$\bbr^2$  are invariant under the transitive group $G=U(2)\le
\Aut(E)$,  to study the $G$-invariant property we only need to work
out its restriction to $z=0$ in $E$.

Using \eqref{bbr} and the convention \eqref{U}, we view $\bbr$
as a function on $U\sub Y$. We consider
\begin{eqnarray*}
\begin{aligned}
i\partial\bar\partial \bbr^2=&i(\mid u\mid ^2+\mid v\mid
^2)\,dz\wedge d\bar
z+i(1+\mid z\mid ^2)(du\wedge d\bar u+dv\wedge d\bar v)\\
&\qquad +iz\bar u \, du\wedge d\bar z+i\bar z u\,dz\wedge d \bar u+iz\bar
v\,dv\wedge d\bar z+iv\bar z \,dz\wedge d\bar v.
\end{aligned}
\end{eqnarray*}
Restricting to $0$, and introducing
\begin{equation*}
\lambda_1=dz,\ \ \lambda_2=\frac{\bar u du+\bar vdv}{\sqrt{\mid
u\mid ^2+\mid v\mid ^2}},\ \  \lambda_3=\frac{v du-udv}{\sqrt{\mid
u\mid ^2+\mid v\mid ^2}}\and \lambda_{k\bar
l}=i\lambda_k\wedge\lambda_{\bar l}\,,
\end{equation*}
we obtain
\begin{equation}\lab{201}
i\partial\bar\partial \bbr^2|_{z=0}= \bbr^2\lambda_{1\bar
1}+\lambda_{2\bar 2}+ \lambda_{3\bar 3}.
\end{equation}
For the same reason, $i\partial \bbr^2\wedge \bar\partial \bbr^2$ is
also $G$-invariant, and has the form
\begin{equation}\lab{202}
i\partial \bbr^2\wedge \bar\partial
\bbr^2|_{z=0}=\bbr^2\lambda_{2\bar 2}.
\end{equation}

\begin{defi}\lab{co0}
Let $f_0=\frac{3}{2} (\bbr^2)^{\frac 2 3}$. The two-form
$i\partial\bar\partial f_0$ is the K\"ahler  form of Candelas-de la Ossa's metric on $\neib\setminus E$. It is $G$-invariant.
\end{defi}

We denote this metric by $\omega_{\co,0}$; call it the CO-metric. In explicit form,
\begin{equation}\lab{203}
\omega_{\co,0}|_{z=0}=(\bbr^2)^{\frac 2 3}\lambda_{1\bar
1}+2/3\ (\bbr^2)^{-\frac 1 3}\lambda_{2\bar 2}+(\bbr^2)^{-\frac 1
3}\lambda_{3\bar 3}.
\end{equation}

\vskip3pt

Our next step is to modify $\omega$ using the CO-metric near $E$. For this, we need to
select a cut off function $\chi(s)$.

\begin{lemm} \lab{lem1} There is a constant $C_1$ such that for  any  sufficiently large $n$,
we can find a smooth function $\chi: [0,\infty)\to\RR$ such that
\begin{enumerate}
\item $\chi(s)=s$ when $s\in [0,2^{\frac 4 3}]$;
\item $\chi'(s)\geq
-C_1n^{-\frac {11} {3}}$ and $2\chi'(s)+s\chi''(s)\geq -C_1n^{-\frac
{11}{3}}$ when $s\in [2^{\frac
4 3},(n-1)^{\frac  4 3}]$;
\item $\chi'(s)\geq -C_1n^{-\frac 5 3}$ and
$2\chi'(s)+s\chi''(s)\geq -C_1n^{-\frac 5 3}$ when $s\in [(n-1)^{\frac 4 3},n^{\frac 4 3}]$;
\item $\chi$ is constant when $s\geq
n^{\frac 4 3}$.
\end{enumerate}
\end{lemm}

\begin{proof}
We first construct a $C^2$-function $\chi$ that satisfies the
required properties. We let $c_1= 2^{\frac 4 3}$; we define
$$\chi(s)=s, \quad \text{for}\ s\in[0,c_1].
$$
We consider
$\phi(s)=c_1+(s-c_1)-(s-c_1)^3$; $\chi$ and $\phi$ have identical
derivatives up to second order at $s=c_1$.

We let $c_2$ to be the (unique) element in $[c_1,\infty)$ so that
$2\phi'(c_2)+c_2\phi''(c_2)=0$. This way,
$\phi'(s)> 0$ and $ 2\phi'(s)+s\phi''(s)\geq 0$ for $s\in[c_1,c_2]$.
We define
$$\chi(s)=\phi(s), \quad \text{for}\ s\in [c_1,c_2].
$$

Next, we pick $c_3=(n-1)^{\frac 4 3}$; $c_3> c_2$ for $n$ large. We define
$$\chi(s)=\chi(c_2)+c_2\chi'(c_2)-{c_2^2\chi'(c_2)}\cdot{s\upmo},\quad
\text{for}\ [c_2,c_3].
$$
One checks that for $s\in[c_2,c_3]$,
$\chi'(s)>0$ and $2\chi'(s)+s\chi''(s)=0$.

To extend $\chi$ to $[c_3,c_4]$ with $c_4=n^{\frac{4}{3}}$, we let
$$\psi(s)=a_0+a_1(s-c_3)+a_2(s-c_3)^2+a_3(s-c_3)^3;
$$
we choose $a_i$ so that $\psi(c_3)=\chi'(c_3)$,
$\psi'(c_3)=\chi''(c_3)$ and $\psi(c_4)=\psi'(c_4)=0$. Solving
explicitly and using $\tau=c_2^2\chi'_2(c_2)$, we get
$$a_0={\tau}c_3^{-2},\
a_1=-2\tau c_3^{-3},\
a_2=\frac{\tau(4c_4-7c_3)}{c_3^3(c_4-c_3)^2},\
a_3=\frac{2\tau(2c_3- c_4)}{c_3^{3}(c_4-c_3)^{3}}.
$$
Using the explicit form of $c_3$ and $c_4$, we see that there is a
constant $C_1$ independent of $n$ so that for large $n$, $-C_1n^{-\frac {10} 3}\leq
a_2< 0$ and $0<a_3\leq C_1n^{-\frac {11}{3}}$. Therefore, over
$[c_3,c_4]$ we have
$\psi(s)\geq -C_1n^{-\frac 5 3}$ and $2 \psi(s)+s\psi'(s)\geq
-C_1n^{-\frac 5 3}$.
We define
$$\chi(s)=\int_{c_3}^s\psi(\tau)d\tau+\chi(c_3),\quad\text{for}\  s\in [c_3,c_4],
$$
and define $\chi$ to be a constant function over $[c_4,\infty)$.

In the end, after a small perturbation of the function $\chi$, we
obtain a smooth function that satisfies the requirements stated.
This proves the Lemma.
\end{proof}

 From now on, we let $n$ be a large integer satisfying the conclusion of
Lemma \ref{lem1}. We introduce some auxiliary functions depending on $n$.
We will use subscript $n$ to emphasize their dependence on $n$. Later we will drop the
subscript $n$ when $n$ is fixed.

We set $\bs_n=n^{\frac 4 3}(\bbr^2)^{\frac 2 3}$
and continue to denote $f_0=\frac 3 2(\bbr^2)^{\frac 2 3}$, both are
functions of $(z,u,v)$.  Using the function $\chi$, we construct a
$d$-closed real $(2,2)$-form on $U\setminus E$:
\begin{equation*}
\Phi_n=\frac{3}{2}\, i\partial\bar\partial\Bigl(n^{-\frac 4
3}\chi(\bs_n)(i\partial\bar\partial f_0)\Bigr);
\end{equation*}
since $\bbr$ is smooth on $U\setminus E$, it is well-defined.
Expanding,
\begin{equation*}
\Phi_n=\chi'(\bs_n)(i\partial\bar\partial f_0)\wedge
(i\partial\bar\partial f_0)+2/3\  n^{\frac 4 3}(\bbr^2)^{-\frac 2
3}\chi''(\bs_n)(i\partial \bbr^2\wedge\bar\partial
\bbr^2)\wedge(i\partial\bar\partial f_0).
\end{equation*}
Restricting $\Phi_n$ to $z=0$ in $E$, from (\ref{202}) and
(\ref{203}), we get
\begin{eqnarray*}
\begin{aligned}
n^{\frac 2 3}\Phi_n|_{z=0}=& 2/3\
\bigl(2\chi'(\bs_n)+\bs_n\chi''(\bs_n)\bigr)\bs_n^{\frac 1 2}\lambda_{1\bar
1}\wedge \lambda_{2\bar 2}+2\chi'(\bs_n)\bs_n^{\frac 1 2}\lambda_{1\bar
1}\wedge \lambda_{3\bar
3}\\
&+2/3\ \bigl(2\chi'(\bs_n)+\bs_n\chi''(\bs_n)\bigr)\bs_n^{\frac 1
2}\bbr^{-2}\lambda_{2\bar 2}\wedge \lambda_{3\bar 3}.
\end{aligned}
\end{eqnarray*}

\begin{lemm}\lab{lem2}
The $(2,2)$-form $\Phi_n$ satisfies:
\begin{enumerate}
\item over {$\neib(\frac 2 n)\setminus E$}, $\displaystyle
\Phi_n=\omega_{\co,0}^2$ is positive;
\item over $\neib\setminus\neib(\frac{2}{n})$, there is a constant
$C_2$ such that for sufficiently large $n$,
\begin{equation*}
n^{\frac 2 3}\Phi_n|_{z=0}\geq -C_2n^{-1}\sum_{k\not= j}\lambda_{k\bar
k}\wedge\lambda_{j\bar j};
\end{equation*}
\item   $\Phi_n$  has compact support (contained in) $U$.
\end{enumerate}
\end{lemm}

\begin{proof}
This follows from Lemma \ref{lem1}.
\end{proof}

Since $\Phi_n$ has compact support (contained in) $
U$; using extension by zero, we can view it as a global form on
{$Y\setminus E$}. \vsp

Next we investigate the restriction of $\omega$ to $\neib$. We
let $\iota: E\to Y$ be the inclusion and consider the restriction
(pull back) $\omega|_E=\iota\sta\omega$; it is a K\"ahler metric on
$E$. With $\pi: \neib\to E$ the tautological projection induced by
the bundle structure of $L\sumt$, the form
$$\tilde \omega_E=\pi\sta(\omega|_E)
$$
is a closed semi-positive $(1,1)$-form on $\neib$.

\begin{lemm} There is a smooth function $h$ of $\neib$ such that
$
\omega|_{\neib}=\tilde\omega_E+i\partial\bar\partial h.
$
\end{lemm}

\begin{proof}
Since $[\omega|_U]=[\tilde \omega_E]\in H_{dR}^2(\neib,\mathbb R)$,
there exists a real 1-form $\alpha$ such that $
\omega|_U-\tilde\omega_E=d\alpha$. Since $\alpha$ is real, we can
write $\alpha=\beta+\bar\beta$ for $\beta$ a $(0,1)$-form. Therefore
from
\begin{equation*}
\omega|_U-\tilde\omega_E=\partial\bar\beta+(\partial\beta+\dbar\bar\beta)+\dbar\beta,
\end{equation*}
we obtain $\dbar\beta=0$.

We now prove that the Dolbeault cohomology group $H^{0,1}_{\dbar}(U,\CC)=0$.
Let $0,\infty\in E$ be the standard $0$ and $\infty$ in $\Po$ using the isomorphism
$E\cong \Po$ fixed at the beginning of this section. Continue to denote by $\pi: U\to E$
the projection, we introduce open subsets
$U_+=U\setminus \pi\upmo(\infty)$, $U_-=U\setminus \pi\upmo(0)$ and $B=U_-\cap U_+$.
Following the argument leading to \eqref{bbr}, $U_+\sub\CC^3$ is the domain $\{(z,u,v)\mid \bbr(z,u,v)^2<1\}$.
Since $\bbr(z,u,v)^2$ is pluri-subharmonic, $U_+$ is Levi-pseudo-convex; thus is a domain of holomorphy.
Applying the Dolbeault theorem \cite[Thm 6.3.1]{Krantz},
$H^{0,1}_{\dbar}(U_+,\CC)=0$. For the same reason,
$H^{0,1}_{\dbar}(U_-,\CC)=0$.

Let $\gamma\in H^{0,1}_{\bar\partial}(U,\mathbb C)$.
Then there exist functions $h_+$ on $U_+$ and $h_-$ on $U_-$ such that
$\gamma|_{U_+}=\bar\partial h_+$ and $ \gamma|_{U_-}=\bar\partial h_-$.
Thus $h_0=(h_+-h_-)|_B$ is holomorphic (on $B$). We claim that we can find holomorphic
$a_+$ on $U_+$ and holomorphic $a_-$ on $U_-$ so that $(a_+-a_-)|_B=h_0$;
therefore $h_+-a_+$ on $U_+$ and $h_--a_-$ on $U_-$ patch along $B$ to form a smooth function
$\tilde h$ on $U$ so that $\dbar \tilde h=\gamma$. This will prove $H^{0,1}_{\dbar}(U,\CC)=0$.

We now prove the claim. We keep the embedding $U_+\sub \CC^3$ mentioned;
using $B\sub U_+$ we have the induced embedding $B\sub U_+\sub\CC^3$.
Let $h_0$ be the holomorphic function on $B$ mentioned. Since for any $c\in \CC\sta$ the slice
$B\cap\{ z=c\}$ is a polydisk in $\CC^2$, $h_0$ has a power series expansion
$h_0=\sum_{i,j\ge 0} a_{ij}(z) u^i v^j$, where $a_{ij}(z)$ are holomorphic on $\CC\sta$.
Using the Laurent series expansions, we can write $a_{ij}(z)=a_{ij}^+(z)+a_{ij}^-(z\upmo)$
so that $a_{ij}^\pm(z)$ are holomorphic on $\CC$ and $a_{ij}^-(0)=0$. (Such
decompositions are unique.)
We let $h_0^+=\sum_{i,j} a^+_{ij}(z) u^i v^j$ and $h_0^-=\sum_{i,j} a^-_{ij}(z\upmo)u^iv^j$.
Using the Cauchy integral formula and applying power series convergence criterion, one checks that
$h_0^+$ extends to a holomorphic function on $U_+$.

It remains to show that $a^-$ extends to a holomorphic function on $U_-$. For this, we use that $U\sub L^{\oplus 2}$
and $L$ is the degree $-1$ line bundle on $\Po\cong E$.
Thus we can embed $U_-\sub \CC^3$ via coordinates $(z',u',v')$
such that the transition function from $U_+$ to $U_-$ is
\begin{equation}
\label{tran}(z',u',v')=(z\upmo, uz, vz).
\end{equation}
(Note that $u\equiv 1$ transforms to $u'=1/z'$, which has a simple pole at
$z'=0$.) Thus $h_0^-=\sum_{i,j} a^-_{ij} (z')(z')^{i+1}(u')^i(v')^j$.
Since $a_{ij}^-(0)=0$, $h_0^-$ converges on $B$ implies that it extends to a holomorphic
function on $U_-$. This proves the claim; hence $H^{0,1}_{\dbar}(U,\CC)=0$.

Because
$H^{0,1}_{\dbar}(\neib,\mathbb C)=0$,
we can find a function $g$ on $\neib$ such that $\beta=\dbar{g}$.
Therefore letting $h=-i(g-\bar{g})$,
$\omega|_U-\tilde\omega_E=i\partial\dbar h$.
\end{proof}

Since
$i\partial\bar\partial h|_E=
\iota\sta(\omega-\tilde\omega_E)=0$, the restriction
$h|_E=\text{const}.$. Thus by subtracting a constant from $h$ we can
assume that $h|_E=0$. Next, using the open $U_+=U\setminus  \pi\upmo(\infty)$
and the embedding $U_+\sub\CC^3$, we introduce directional derivatives:
\begin{equation}\label{direction}
\ba=\frac{\partial}{\partial u}(h|_{U_+})|_{E-\infty}\and \bb=\frac{\partial}{\partial v}(h|_{U_+})|_{E-\infty}.
\end{equation}
Using the embedding $U_-\sub\CC^3$ and the transition function \eqref{tran}, one sees that
the smooth function on $U_+$ defined via
\begin{equation}\label{h1}
h_1|_{U_+} := \ba u+\bar \ba \bar u+\bb v +\bar\bb\bar v
\end{equation}
extends to a smooth function on $U$ that is
$\RR$-linear along the fibers of $\pi:U\to E$; we denote this extension by $h_1$.

Using $h_1$, we now introduce another $(2,2)$-form. We let $h_2=h-h_1$. We pick a
decreasing function $\sigma(s)$ that takes value $1$ when $0\leq
s\leq 1$ and vanishes when $s\geq 4$. We set $\bt_n=n^2\bbr^2$, which
is a function of $(z,u,v)$. Since $\sigma(\bf t)$ has
compact support (contained in) $\overline {U(\frac 2 {n^2})}$, using
extension by zero, we can view it as a function on $Y$ and then use
it to define a real $(2,2)$-form on $Y$:
\begin{equation*}
\Psi_n=\omega^2-i\partial\bar\partial\bigl(\sigma(\bt_n)\cdot
h_2\cdot\bigl(2\tilde\omega_E+i\partial\bar\partial(2h_1+h_2)\bigr)\bigr)
-i\partial\bar\partial\bigl(\sigma(\bt_n)\cdot h_1\cdot
i\partial\bar\partial h_1\bigr).
\end{equation*}
This form satisfies
$$d\Psi_n=0,\qquad \Psi_n|_{Y\setminus \neib(\frac 2 n)}=\omega^2\and
\Psi_n|_{\neib(\frac 1 n)}=0.
$$
Here the first and second follows from the definitions of $\Psi_n$ and
$\sigma(\bt_n)$; the third follows from $i\partial\bar\partial
h_1\wedge \tilde\omega_E=0$.

We now add a multiple of the compactly supported form $\Phi_n$ to
$\Psi_n$:
\begin{equation*}
\Omega_0=\Psi_n+C_0n^{\frac 2 3}
\Phi_n,\quad C_0>0.
\end{equation*}
We emphasize that the form $\Omega_0$ depends on the constant $C_0$
and the integer $n$. We shall specify their choices later.

\begin{lemm}\lab{l5}
The real $(2,2)$-form $\Omega_{0}$ is $d$-closed;
\begin{enumerate}
\item restricting to ${\neib(\frac 1 n)\setminus{E}}$,
$\Omega_{0}|_{{\neib(\frac 1 n)\setminus{E}}}=C_0n^{\frac 2 3}\omega_{\co,0}^2$;
\item restricting to ${Y\setminus \neib}$,
$\Omega_{0}|_{{Y\setminus \neib}}=\omega^2$.
\end{enumerate}
Further, for sufficiently large $C_0$, there is a constant $n(C_0)$
such that for $n\geq n(C_0)$, $\Omega_{0}>0$.
\end{lemm}

\begin{proof}
Because $\Phi_n$ and $\Psi_n$ are both  $d$-closed, $\Omega_0$ is
d-closed. We show that $\Omega_0>0$. By the definitions of $\Phi_n$ and $\Psi_n$,
$$\Omega_0|_{X\setminus U}=\Psi_n|_{X\setminus U}=\omega^2>0
\and \Omega_0|_{U(\frac 1 n)}=C_0n^{\frac 2 3}\Phi_n|_{U(\frac 1
n)}=C_0n^{\frac 2 3}\omega_{\co,0}^2>0.
$$
So
we only need to check the positivity of $\Omega_{0}$ over
$\neib\setminus\neib(\frac 1 n)$. We first look at the region
$\neib(\frac 2 n)\setminus\neib(\frac 1 n)$. Within this region,
\begin{equation}\lab{205}
\begin{aligned}
\Psi_n =&(1-\sigma(\bt_n))\omega^2-i\bigl( h_1\partial\bar\partial
\sigma(\bt_n)+\partial\sigma(\bt_n)\wedge\bar\partial h_1+\partial
h_1\wedge\bar\partial\sigma(\bt_n)\bigr)\wedge i\partial\bar\partial
h_1\\
&-i\bigl( h_2 \partial\bar
\partial\sigma(\bt_n)+\partial \sigma(\bt_n)\wedge \bar\partial
h_2+\partial h_2\wedge\bar\partial\sigma(\bt_n)\bigr)\wedge
\bigl(2\tilde\omega_E+i\partial\bar\partial(2h_1+h_2)\bigr).
\end{aligned}
\end{equation}
Since $1-\sigma(\bt_n)\geq 0$, the first term is non-negative. For the
other two terms, because $E$ is covered by $D=\{|z| \leq 2\}$
and $D'=\{| z| \geq 1\}$, we only need to investigate the
positivity over $D$ and $D\pri$ separately. Because the discussion
is similar, we shall deal with $D$ now.

To begin with, we fix a small  $\delta>0$ (to be determined later). We  consider
$V_{\delta}=\pi^{-1}(D)\cap\neib(\delta)$. Over $V_{\delta}$, the
second term in \eqref{205} is
$$-i\bigl(h_1\partial\bar\partial
\sigma(\bt_n)+\partial\sigma(\bt_n)\wedge\bar\partial h_1+\partial
h_1\wedge\bar\partial\sigma(\bt_n)\bigr)\wedge
i\partial\bar\partial h_1,
$$
which, after expanding, becomes
\begin{equation*}\lab{206}
-n^2h_1\bigl(\sigma'i\partial\dbar \bbr^2+\bt_n\sigma''r^{-2}i\partial
\bbr^2\wedge\dbar \bbr^2\bigr)\wedge i\partial\dbar h_1 -\bt_n\sigma'
\bbr^{-2}(i\partial \bbr^2\wedge \dbar h_1+i\partial h_1\wedge \dbar
\bbr^2)\wedge i\partial\dbar h_1.
\end{equation*}
Over the same region, we expand the relevant terms:
$$\partial \bbr^2=\Gamma^{-2}\bbr^{2}\bar z\lambda_1+\Gamma \bbr\lambda_2
$$
for $\Gamma\triangleq(1+\mid z\mid ^2)^{\frac{1}{2}}$, and
\begin{equation*}\lab{207}
i\partial\dbar \bbr^2=\Gamma^{-2}\bbr^2\lambda_{1\bar 1}
+\Gamma^2\cdot (\lambda_{2\bar 2}+\lambda_{3\bar 3})+\Gamma^{-1}\bbr\cdot (\bar
z\lambda_{1\bar 2}+z\lambda_{2\bar 1}).
\end{equation*}

 For simplicity, we use
subindex $z$ and $\bar z$ to denote the partial derivatives with
respect to $z$ and $\bar z$.  For instance, $\ba_{z}=\frac{\partial\ba}{\partial z}$ and $\bb_{z\bar
z}=\frac{\partial^2\bb}{\partial z\partial \bar z}$. We introduce
$$
c_{1\bar 1}=2\text{Re}\,\bl \frac{\ba_{z\bar z}u+\bb_{z\bar
z}v}{\brr}\br,\quad c_{2\bar 1}=\overline{{c}_{1\bar 2}}=\Gamma\cdot\frac{
\ba_{\bar z}u +\bb_{\bar z} v}{\bbr},\quad c_{3\bar
1}=\overline{{c}_{1\bar 3}}=\Gamma\cdot\frac{\ba_{\bar z}\bar
v-\bb_{\bar z}\bar u}{\bbr},
$$
$$
d_{1\bar 2}= \frac{\ba_zu+\bb_zv+\overline{\ba_{\bar z}}\bar
u+\overline{\bb_{\bar z}}\bar v}{\bbr},\quad d_{2\bar 2}=
\Gamma\cdot\frac{ {\bf a}u+{\bf b}v}{\bbr},\quad d_{3\bar 2}=
\Gamma\cdot\frac{{\bf a}\bar v-{\bf b}\bar u}{\bbr},
$$
where $\text{Re}$ is the real part. Following such convention, we
have
\begin{equation*}
\partial h_1=\bbr d_{1\bar 2}\lambda_1+d_{2\bar 2}\lambda_2+d_{3\bar 2}\lambda_3
\end{equation*}
and
\begin{equation*}
i\partial\bar\partial h_1=\bbr c_{1\bar 1}\lambda_{1\bar 1}+c_{2\bar
1}\lambda_{2\bar 1}+c_{1\bar 2}\lambda_{1\bar 2}+c_{3\bar
1}\lambda_{3\bar 1}+c_{1\bar 3}\lambda_{1\bar 3}.
\end{equation*}

To simply further, we introduce
\begin{eqnarray*}
\begin{aligned}
\alpha_{1\bar 2}=&\overline{\alpha_{2\bar 1}}=-nh_1\sigma'
\Gamma^2c_{2\bar 1}+\bt_n^{\frac 1 2
}\sigma'\Gamma c_{3\bar 1}\overline{ d_{3\bar 2}};\\
\alpha_{2\bar 2}=&-nh_1\bt_n^{\frac 1 2}\sigma'\Gamma^2c_{1\bar 1}
+2\bt_n\sigma'\Gamma^{-2}\text{Re}(zc_{1\bar 3}d_{3\bar 2});\\
\alpha_{2\bar 3}=&\overline{\alpha_{3\bar 2}}=nh_1\bt_n^{\frac 1
2}(\sigma'+\bt_n\sigma'')\Gamma^{-1}\bar z c_{3\bar
1}+\bt_n\sigma'\Gamma^{-2}\bigl(\bar z c_{3\bar
1}\overline{d_{2\bar 2}}+zc_{1\bar 2}d_{3\bar 2}\bigr)\\
&+\bt_n\sigma'\Gamma(c_{3\bar 1}d_{1\bar 2}-c_{1\bar 1}d_{3\bar
2});\\
\alpha_{1\bar
3}=&\overline{\alpha_{3\bar1}}=-nh_1(\sigma'+\bt_n\sigma'')\Gamma^2c_{3\bar
1}-\bt_n^{\frac 1 2 }\sigma'\Gamma(2c_{3\bar
1}\text{Re}d_{2\bar 2}-c_{2\bar 1}d_{3\bar 2});\\
\alpha_{3\bar 3}=&-nh_1\bt_n^{\frac 1
2}(\sigma'+\bt_n\sigma'')\bigl(\Gamma^2c_{1\bar 1}-2\Gamma^{-1}\text{Re}(zc_{1\bar 2})\bigr)\\
&-2\bt_n\sigma'\Gamma\bigl( c_{1\bar 1}\text{Re}d_{2\bar
2}-\text{Re}(c_{2\bar 1}d_{1\bar 2})-\Gamma^{-3}\text{Re}(zc_{1\bar
2}d_{2\bar 2})\bigr).
\end{aligned}
\end{eqnarray*}

Because for $\bbr$ small, $|u| , | v| \leq 2\bbr$, we can find a
constant  $C_3$ depending on $\delta$ so that
\begin{equation*}
 | c_{i\bar j}|,\  |  d_{i\bar j}| \leq C_3\quad \text{for}\ (z,u,v)\in V_\delta.
\end{equation*}
To control the terms $\alpha_{i\bar j}$, we need to bound the term
$nh_1$. For this, since $ n^{-1}\leq \bbr< 2 n^{-1}$ over
$U(\frac{2}{n})\setminus U(\frac{1}{n})$, the term $nh_1$ is bounded
from above uniformly over $V_\delta\cap\bl U(\frac{2}{n})\setminus
U(\frac{1}{n})\br$ for $n>\frac{2}{\delta}$.  Thus $\delta$ must be sufficiently small, which
can be determined accordingly to Lemma \ref{lem1}.
Therefore, enlarging $C_3$ if necessary and for
$n>\frac{2}{\delta}$, we have
$$ |  \alpha_{i\bar
j} |  \leq C_3,\quad (z,u,v)\in
V_{[\frac{1}{n},\frac{2}{n}]}\triangleq \pi\upmo(D)\cap\bl
U(2/n)\setminus U(1/n)\br.$$

Finally, we introduce $\Lambda_{i\bar j}$
$$\Lambda_{i\bar
j}=(i\lambda_{k}\wedge\lambda_{\bar k})\wedge
(i\lambda_{l_1}\wedge\lambda_{\bar l_2})=\lambda_{k\bar
k}\wedge\lambda_{l_1\bar l_2}\ \ \text{for}\ \
\{i,k,l_1\}=\{j,k,l_2\}=\{1,2,3\}.
$$
Simplifying using the
notations introduced, the expression
\begin{eqnarray}\lab{208}
\begin{aligned}
&-i\bigl(h_1\partial\bar\partial
\sigma(\bt_n)\partial\sigma(\bt_n)\wedge\bar\partial h_1+\partial
h_1\wedge\bar\partial\sigma(\bt_n)\bigr)\wedge i\partial\bar\partial h_1\\
&\quad =n\sum_{l=2,3}(\alpha_{1\bar l}\Lambda_{1\bar l}+\alpha_{l\bar
1}\Lambda_{l\bar 1})+\sum_{k,l=2,3}\alpha_{k\bar l}\Lambda_{k\bar
l}.
\end{aligned}
\end{eqnarray}

We now look at the third term in (\ref{205}). This time we consider
\begin{equation*}
\begin{aligned}
&-i(h_2 \partial\bar
\partial\sigma(\bt_n)+\partial \sigma(\bt_n)\wedge \bar\partial
h_2+\partial h_2\wedge\bar\partial\sigma(\bt_n))|_{V_\delta}\\
&\quad =-\bbr^{-2}h_2(\bt_n\sigma'i\partial\dbar \bbr^2-\bt_n^2\sigma''
\bbr^{-2}i\partial \bbr^2\wedge \dbar \bbr^2)
-\bt_n\sigma'\bbr^{-2}(\partial \bbr^2\wedge \dbar h_2+\partial
h_2\wedge \dbar \bbr^2). \end{aligned} \end{equation*} Since
restricting to $E$ the partial derivatives of $h_2$ with respect to
$u$ and $v$ are zero, when $\bbr$ is small, $|  h_2|
=O(\bbr^2)$ and $| \partial h_2| =O(\bbr)$. Also notice that
the mixed term such as $\lambda_{2\bar 3}$ can be controlled by
$\lambda_{2\bar 2}$ and $\lambda_{3\bar 3}$.
Therefore for $n>\frac{2}{\delta}$, over
$V_{[\frac{1}{n},\frac{2}{n}]}$ we have
\begin{equation*}
C_3\sum_{k=1}^3\lambda_{k\bar k}\geq -i(h_2
\partial\bar
\partial\sigma(\bt_n)+\partial \sigma(\bt_n)\wedge \bar\partial
h_2+\partial h_2\wedge\bar\partial\sigma(\bt_n))\geq
-C_3\sum_{k=1}^3\lambda_{k\bar k}.
\end{equation*}
Therefore the third term in (\ref{205}) can be controlled by
$-C_3\sum_{k}\Lambda_{k\bar k}$.
Inserting  this and (\ref{208}) into (\ref{205}), we get
\begin{eqnarray*}\lab{012504}
\Psi_n
\geq n\sum_{l=2,3}(\alpha_{1\bar l}\Lambda_{1\bar l}+\alpha_{l\bar
1}\Lambda_{l\bar 1})+\sum_{k,l=2,3}\alpha_{k\bar l}\Lambda_{k\bar
l}-C_3\sum_{k=1}^3\Lambda_{k\bar k}.
\end{eqnarray*}
On the other hand by a directly calculation, we have
\begin{equation*}\lab{2000}
\begin{aligned}
n^{\frac 2 3}\Phi_n|_{V_{[\frac 1 n,\frac 2 n]}}=& 4/3\ \bt_n^{-\frac 2
3}\Gamma^4n^2\Lambda_{1\bar 1}+4/3\  \bt_n^{-\frac 1 6}n\Gamma\bar
z\Lambda_{2\bar 1}+ 4/3\ \bt_n^{-\frac 1 6}n\Gamma z\Lambda_{1\bar
2}\\
+&2\bt_n^{\frac 1 3}(1-3^{-1}\Gamma^{-2}| z| ^2)\Lambda_{2\bar
2}+ 4/3\ \bt_n^{\frac 1 3}(1-\Gamma^{-2}| z| ^2)\Lambda_{3\bar
3}.
\end{aligned}
\end{equation*}
Combining above two, over $V_{[\frac{1}{n},\frac{2}{n}]}$ we
finally obtain
\begin{eqnarray*}
\begin{aligned}
\Omega_{0}&\geq\Bigl(\frac {4\Gamma^4n^2}{3\bt_n ^{\frac 2
3}}C_0-C_3\Bigr)\Lambda_{1\bar 1}+n\Bigl( \frac {4\Gamma z} {3
\bt_n^{\frac 1 6}} C_0+\alpha_{1\bar 2}\Bigr)\Lambda_{1\bar
2}+n\alpha_{1\bar 3}\Lambda_{1\bar 3}\\
+n&\Bigl(\frac {4\Gamma\bar z} {3 \bt_n^{\frac 1 6}}C_0+\alpha_{2\bar
1}\Bigr)\Lambda_{2\bar 1}+\Bigl(2\bt_n^{\frac 1 3}\bigl(1-\frac {|
z| ^2} {3\Gamma^2}\bigr)C_0-C_3+\alpha_{2\bar
2}\Bigr)\Lambda_{2\bar 2}+\alpha_{2\bar 3}\Lambda_{2\bar
3}\\
&+n\alpha_{3\bar 1}\Lambda_{3\bar 1}+\alpha_{3\bar 2}\Lambda_{3\bar
2}+ \Bigl(\frac 4 3 \bt_n^{\frac 1 3}\bigl(1-\frac {| z| ^2}
{\Gamma^2} ) C_0-C_3+\alpha_{3\bar 3}\Bigr)\Lambda_{3\bar 3}.
\end{aligned}
\end{eqnarray*}

We now prove that we can find a sufficiently large constant
$C_0$
so that for any $n>\frac 2 \delta$, the right hand
side of the above inequality is positive. We let
$e_{ij}$ be the coefficient of the term $\Lambda_{i\bar j}$ in the
above inequality. To prove the mentioned positivity, we only need to
check that under the stated constraint, the three minors of the
$3\times 3$ matrix $[e_{ij}]$ are positive:
$$e_{11}>0, \quad \det [e_{ij}]_{1\leq i,j\leq 2}>0,\quad \det [e_{ij}]_{3\times 3}>0.
$$

We recall that $\bt_n=n^2\bbr^2$ and $\Gamma=(1+| z| ^2)^{\frac
1 2}$. So in the region $V_{[\frac{1}{n},\frac{2}{n}]}$,
$1\leq \bt_n<2$ and $1\leq \Gamma\leq \sqrt{5}$. Therefore by
expanding the determinants, we see immediately that they are all
positive for $n$ positive and $C_0$ large enough. We fix
such a $C_0$ in the definition of $\Omega_0$. Therefore, for
any $n>\frac 2 \delta$, the form $\Omega_0$ is positive in the
region $U(\frac 2 n)\setminus U(\frac 1 n)$.

It remains to consider the region $\neib\setminus \neib(\frac 2 n)$.
Over this region, we shall prove that $\Omega_0$ is positive when
$n$ is large enough. For this purpose, we will use the smooth
homogenous Candelas-de la Ossa metric \cite{CO} on $\neib$:
\begin{equation}\lab{cos}
 \omega_{\co}=i\partial\dbar f(\bbr^2)+i\partial\dbar\log(1+|
z| ^2),
\end{equation}
where $f$ is defined via $f'=\bbr^{-2}\eta$ for
$\eta^2(\eta+3/2)=\bbr^4$. Explicitly,
\begin{equation}\lab{209}
\begin{aligned}
\omega_{\co}|_{z=0}=(\eta+1) \lambda_{1\bar 1}+\frac 2 3
\frac{(\eta+\frac 3 2)^{\frac 1 2}}{\eta+1}\lambda_{2\bar 2}+\frac 1
{(\eta+\frac 3 2)^{\frac 1 2}} \lambda_{3\bar 3}.
\end{aligned}
\end{equation}
By simple estimate,
\begin{equation*}
\omega_{\co}^2|_{z=0}\geq \frac 1 3\sum_{k\not = j}\lambda_{k\bar
k}\wedge\lambda_{j\bar j}
\end{equation*}
Comparing with Lemma \ref{lem2} (2), since both $\omega_{\co}^2$ and
$\Phi_n$ are homogeneous, over $\neib\setminus \neib(\frac 2 n)$ we
get
\begin{equation*}
n^{\frac 2 3}\Phi_n\geq -3C_2n^{-1}\omega_{\co}^2.
\end{equation*}
Therefore, over $\neib\setminus\neib(\frac 2 n)$,
\begin{eqnarray*}
\Omega_{0}\geq \omega^2-3C_0C_2n^{-1}\omega_{\co}^2.
\end{eqnarray*}
This proves that for the fixed $C_0$ and $C_2$, we can choose $n$
big enough so that the real $(2,2)$-form $\Omega_{0}$ is positive
over $\neib\setminus\neib(\frac 2 n)$. This proves the lemma.
\end{proof}

The closed $(2,2)$-form $\Omega_0$ is positive  $(2,2)$-form on
$Y\setminus E$. From \cite{Mic}, there is a positive $(1,1)$-form
$\omega_0$ on $Y\setminus E$ such that $\omega_0^2=\Omega_0$. This
proves

\begin{prop}\lab{6}
For the open subset $E\sub U\sub Y$ chosen
and for sufficiently large $C_0$ and $n$, we can find a balanced
metric $\omega_0$ over $Y\setminus E$ such that 
\begin{enumerate}
\item restricting to $Y\setminus \neib(1)$: $\omega_0=\omega$;
\item restricting to $\neib(\frac 1 n)\setminus E$: $\omega_0=C_0^{\frac 1 2}n^{\frac 1 3}\omega_{\co,0}$;
\item and restricting to
$\neib(1)\setminus\neib(\frac 1 n)$: $\omega_0^2$ is $\partial\dbar$-exact.
\end{enumerate}
\end{prop}

Let $Y$ be a Calabi-Yau manifold and $\omega$ its K\"ahler metric.
Let $\bE\subset Y$ be a union of mutually disjoint $(-1,-1)$-curves
$E_i\sub Y$, $1\le i\le l$. For each $E_i\sub Y$ we choose an open neighborhood
$E_i\sub U_i\sub Y$ as given by Proposition \ref{6}, and form $U_i(c)$ accordingly.
We let $\bU=\cup_{i=1}^l U_i$ and let $\bU(c)=\cup_{i=1}^l U_i(c)$. We have

\begin{coro}\lab{c7}
Let the notation be as stated. Then the Proposition \ref{6} holds
true with $U$ and $U(c)$ replaced by $\bU$ and $\bU(c)$, respectively.
\end{coro}

\begin{proof}
Since the proof of Lemma \ref{l5} is by modifying $\omega^2$ within
the open neighborhood $E\subset U\subset Y$, if we choose
$U_i$ to be mutually disjoint, then we can
modify $\omega^2$ the same way within $\bU$ to obtain the
desired metric $\omega_0$. Note that from the proof of Lemma
\ref{l5}, we can choose a common $n$ and $C_0$ that work for all $i$.
\end{proof}

Because $Y-\bE=X_{0,\sm}$ (the smooth part of $X_0$), $\omega_0$ is
a smooth balanced metric on $X_{0,\sm}$ that is Candelas-de la Ossa's metric near the singular points of $X_0$.

 \section{constructing balanced metrics on the smoothings}

Assuming the threefold $X_0$ can be smoothed to a family of smooth
Calabi-Yau threefolds $X_t$, in this section we shall show that we
can deform the metric $\omega_0$ to a family of smooth balanced
metrics on $X_t$.

\begin{defi}\label{total}
We say $X_t$ is a smoothing of $X_0$ if there is a smooth four
dimensional complex manifold $\mathcal X$ and a proper holomorphic
projection $\mathcal X\to\Delta$ to the unit disk $\Delta$ in $\CC$
so that the general fibers $X_t=\mathcal X\times_{\Delta}t$ are
smooth and the central fiber $\mathcal X\times_\Delta 0$ is
$X_0$.
\end{defi}




Let $X_0$ be a singular space that is a construction of disjoint $(-1,-1)$-curves; let
$\omega_0$ be the balanced metric on $X_{0,\sm}$ constructed in the
previous section. We suppose $X_t$ is a smoothing of $X_0$
with $\mathcal X$ the total space of the smoothing.

We begin with the local geometry of $\mathcal X$ near a singular
point of $X_0$. Let $p\in X_0$ be any singular point that is the
contraction of $E=\pi\upmo(p)$. Since $X_0$ is a contraction of
$(-1,-1)$-curves in $Y$, from the classification of singularities of
threefolds,
a neighborhood of $p$ in $X_0$ is
isomorphic to a neighborhood of $0$ in
$$ w^2_1+w^2_2+w^2_3+w^2_4=0.
$$
Applying the theorem in \cite{Sh}, a neighborhood of $p$ in the total family
$\cX$ is isomorphic to a neighborhood of $0$ in
$$ w^2_1+w^2_2+w^2_3+w^2_4-t=0 \quad(\text{in}\  \CC^4\times\Delta),
$$
as a family over $t\in\Delta$. (Here the $t$ is linear because $\cX$ is smooth.)
More precisely, for some $\eps>0$ and
for
$$\tilde U=\{(w,t)\in \CC^4\times\Delta_\eps \mid |t|<\eps, \,\norml w\normr <1,\, w^2_1+w^2_2+w^2_3+w^2_4-t=0\},
$$
($\norml w\normr^2 =\sum_{i=1}^4 | w_i| ^2$,) there is a holomorphic map
$$\xi:\tilde U\lra \mathcal X,
$$
commuting with the projections $\tilde U\to \Delta_\eps$ and
$\mathcal X\to\Delta$, so that $\mathcal{U}=\xi(\tilde U)$ is an
open neighborhood of $p\in \mathcal X$ and $\xi$ induces an
isomorphism from $\tilde U$ to $\mathcal{U}\sub \mathcal X$.

We fix such an isomorphism $\xi$; we denote by $\tilde U_t$ the
fiber of $\tilde U$ over $t\in\Delta_\eps$, and denote
$U_t=\xi(\tilde U_t)$, which is an open subset of
$X_t\cap\mathcal{U}$. For any $1>c>0$, we let
$$\tilde\neib(c)=\{(w,t)\in\tilde U\mid  \norml w\normr<c\}\and \neib_t(c)=\xi(\tilde\neib(c))\cap X_t.
$$
This way, for fixed $t$, $U_t(c)$ forms an increasing sequence of
open subsets of $X_t$; the variables $(w_1,\cdots,w_4)$ can
{be viewed} as coordinate functions with the constraint $\sum
w_i^2=t$ understood.

In case $t=0$, we can choose $\xi$ so that the $(w_1,\cdots,w_4)$
relates to the coordinate $(z,u,v)$ of \eqref{U} by
$$ w_1=\frac{v-zu}{R},
\ w_2=\frac{v+zu}{iR}, \ w_3=\frac{u-zv}{iR}, \ w_4=\frac{u+zv}{R},
$$
where the constant $R$ is to be determined momentarily. Hence under $\xi$
the function $\bbr$ introduced in Section 2
coincides with the function { $R\cdot (\parallel
w\parallel)|_{U_0}$.} We then define $\bbr$ on $\mathcal U$ to be
$\bbr=r\circ \xi\upmo$; they are extensions of the similarly
denoted $\bbr$ on $X_0$ used in the previous section. Also, the
punctured opens $U_0(c)\sta= U_0(c)-p$ are isomorphic to the opens
$U(c)-E$ used in the previous section under $\xi$ as well. Since we
need to work with different fibers $X_t$ simultaneously, we shall
reserve the subscript
$U_t(c)$ to denote open subsets in $X_t$. 

We now choose $R$. By choosing $R$ large and rescaling $\omega_0$,
we can assume that for $f_0=\frac 3 2(\bbr^2)^{\frac 2 3}$
\begin{equation}\lab{metric}
\Omega_0|_{\neib_0(1)}=\omega_0^2|_{\neib_0(1)}=i\partial\bar\partial(f_0\cdot
i\partial\bar\partial f_0).
\end{equation}
Here since $f_0$ is understood as a function on { $U_0(1)\subset
X_0$,} the partials
  $\partial$ and $\bar\partial$
are holomorphic and anti-holomorphic differentials of $X_0$.

One more convention we need to introduce before we move on. Note
that $X_0$ has several singular points, say $p_1,\cdots,p_l$,
corresponding to contracting $E_i\sub Y$. For each such $p_i$, we
will go through the same procedure as we did for a general singular
$p\in X_0$ moments earlier to pick an open $p_i\in \mathcal{U}_i\sub
\mathcal X$, an isomorphism $\xi_i: \tilde
U\to\mathcal{U}_i\sub\mathcal X$ and the open subsets
$U_{i,t}(c)\sub X_t$, etc. In fixing the $\xi_i$ for various $p_i$,
{since we can choose a common $C_0$ and $n$ for all $i\in\{1,\cdots,
l\}$ in Corollary \ref{c7}, we can} pick a single large enough $R$
that works for all $U_{i,0}$ so that \eqref{metric} holds over
$U_{i,0}$.

We then form $V=\cup_{i=1}^l \mathcal{U}_i\sub\mathcal X$ and
$V(c)=\cup_{i=1}^l \mathcal{U}_i(c) \sub\mathcal X$. Accordingly, we
let $V_t=V\cap X_t$, let $V_t(c)=V(c)\cap X_t$, and let $\bbr$ be
the function on $V$ whose restriction to each $\mathcal{U}_i\sub V$
is the $\bbr=r\circ\xi\upmo_i$ defined moment earlier. This
procedure gives us a $(2,2)$-form $\Omega_0$ on $X_0$ such that
\begin{equation}\lab{metric2}
\Omega_0|_{V_0(1)}=\omega_0^2|_{V_0(1)}=i\partial\bar\partial(f_0\cdot
i\partial\bar\partial f_0).
\end{equation}
\va

With these preparations, we now study the deformation $X_t$ near
the singular points $p_i\in X_0$. For $c\in (0,1]$, we
introduce
$$X_t[c]=X_t\setminus V_t(c).
$$
For small $t$, $X_t[\ha]$ are diffeomorphic to each other. We fix
diffeomorphisms $\psi_t: X_t[\ha]\to X_0[\ha]$ that depend smoothly on
$t$ and  that $\psi_0=\text{id}$.
The diffeomorphisms $\psi_t$ pull back the form on $X_0[\frac 1 2]$ to
form on $X_t[\ha]$.

We then let $\varrho(s)$ be a (decreasing) cut-off function such
that $\varrho(s)=1$ when $s\leq \frac 5 8$ and $\varrho(s)=0$ when
$s\geq\frac 7 8$. This  function defines a cut off function
$\varrho_0$ on $X_0$ by rule $\varrho_0|_{X_0[1]}=0$,
$\varrho_0|_{V_0(\frac 1 2)}=1$ and $\varrho_0|_{V_0(1)\setminus
V_0(\frac 1 2)}=\varrho(\bbr)$. Then
$$\Omega_0-i\partial\bar\partial\bigl(\varrho_0\cdot f_0\cdot
i\partial\bar\partial f_0\bigr)
$$
is a smooth $(2,2)$-form on $X_0$ with compact support {
contained in} $X_0[\ha]$. In particular, for small $t$,
$$\psi_t^\ast\bigl(\Omega_0-i\partial\bar\partial(\varrho_0\cdot f_0\cdot
i\partial\bar\partial f_0)\bigr)$$ is a form on $X_t[\frac 1 2]$
which compact support lies in it. So we can view this form as the
form defined on $X_t$ by defining it to be $0$ on
$V_t(\frac 1 2)$.

Momentarily, we will use $\partial$ and $\bar\partial$ over $X_t$.
In the remainder of this paper, we will take holomorphic and
anti-holomorphic differentials of functions on $X_t$ for either
$t\ne 0$ or $t=0$. To keep the notation simple, we use the same
$\partial f$ and $\bar\partial f$ to mean $\partial|_{X_t} f$ and
$\bar\partial|_{X_t}f$ on either $t=0$ or $t\ne 0$, depending on
whether $f$ is a function on $X_0$ or $X_t$. We hope the meaning of
$\partial$ and $\bar\partial$ will cause no confusions. \va

In order to construct a positive $(2,2)$-form on $X_t$, we need to
extend the function $f_0(\bbr^2)=\frac 3 2 \bbr^{\frac 4 3}$ defined
in Definition \ref{co0}. For $t\ne 0\in \bigtriangleup_{\epsilon}$,
we define $f_t(s)$ to be the function
\begin{equation}\lab{pot}
f_t(s)=2^{-\frac 1 3}| t|^{\frac 2
3}\int_{0}^{\cosh^{-1}(\frac {s}{\mid
t\mid})}(\sinh2\tau-2\tau)^{\frac 1 3}d\tau, \quad s\in (0,1).
\end{equation}
The functions $f_t$ give Candelas-de la Ossa's metrics (CO-metric).

\begin{defi} The two form
$\omega_{\co,t}=i\partial\bar\partial f_t(\bbr^2)$ is the Ricci-flat
K\"ahler form on $V_t(1)$ constructed by Candelas and de la Ossa.
\end{defi}

For later application, we need to confirm the smooth dependence of the metrics
$\omega_{oc,t}$ on $t$.
We denote by $f_t^{(k)}(s)$ the $k$-th derivative in $s$ of $f_t(s)$.

\begin{lemm}\lab{control}
Let $f_0(s)=\frac{3}{2} s^{\frac{2}{3}}$. Then
\newline
(1). for any $\delta>0$ and $k$, restricting to $s\in[\delta,1]$ the functions $f_t^{(k)}(s)$
converges uniformly to $f_0^{(k)}(s)$ when $t$ goes to zero;
\newline
(2). For any pair $0<\delta'<\delta<\frac 1 4$, there exists a
$\alpha_{\delta'}$ such that when $\mid t\mid<\alpha_{\delta'}$ and
$s\in[\delta',\delta]$, $\frac1 2\leq \frac{ f_t'(s)}{f_0'(s)}\leq
2$ and $\frac 1 2 \leq \frac{f_t''(s)}{f_0''(s)}\leq 2$.
\end{lemm}

\begin{proof}
Since the dependence on $t\in\Delta_\epsilon$ is via its norm, we
shall substitute $|t|$ by the positive real variable $u$ and define $f_u(s)$
as in \eqref{pot} with $t$ replaced by $u>0$.

At first we consider the convergence of $f_u(s)$. Using
\eqref{pot}, we get
$$\lim_{u\to 0}f_u(s)=\frac 3 2 s^{\frac 2 3}\lim_{u\to 0} g_u(s),
$$
where $g_u(s)={\bigl(\bigl(1-\frac{u^2}{s^2}\bigr)^{\frac 1 2}-\frac{u^2}{s^2}\cosh^{-1}(\frac
s u)\bigr)^{\frac 1 3}}{\bigl(1-\frac{u^2}{s^2}\bigr)^{-\frac 1 2}}$.
Since  $\frac{u^2}{s^2}\cosh^{-1}(\frac s u)\sim
\frac{u^2}{s^2}|\ln u|$ when $s\in[\delta,1)$, $g_u(s)$
converges uniformly to $1$ in $[\delta,1)$; so $f_u(s)$ converges
uniformly to $f_0(s)=\frac 3 2s^{\frac 2 3}$.

For the first and the second derivatives. by \eqref{pot}, we compute
\begin{equation*}
f_u'(s)=s^{-\frac 1 3}g_u(s)\and
f_u''(s)=\bigl(-s^{-1}f_u'(s)+\frac 2 3s^{-2}(f'_u(s))^{-2}\bigr)\bigl(1-\frac
{u^2}{s^2}\bigr)^{-1}.
\end{equation*}
So by inspection, over $[\delta,1)$,  $f'_u(s)$  converges uniformly to $f_0'(s)=s^{-\frac 1 3}$
and $f_u''(s)$ converges
uniformly to $f_0''(s)=-\frac 1 3 s^{-\frac 4 3}$.

Since for any $k>0$, the $k$-th derivative of
$(1-\frac{u^2}{s^2})^{-1}$ converges uniformly to the zero function over
$s\in[\delta,1)$, applying induction proves the remainder cases of (1).

The second part of the Lemma follows form the explicit expressions of $f_u'(s)$
and $f_u''(s)$. This proves the Lemma.
\end{proof}

Our next step is to extend $\Omega_0$ to nearby fibers so that near
the singular point it equals to the CO-metrics $\omega_{\co,t}$. To
this end, we define
\[
\varrho_t(z)=\left\{\begin{array}{ll} (\psi_t^\ast\varrho_0)(z)\ \ & z\in X_t[\frac 1 2],\\
1 \ \ & z\in V_t(\frac 1 2);
\end{array}
\right.
\]
we define
\begin{equation}\lab{Phi}
\Theta_t=\psi_t ^\ast\bigl(\Omega_0- i\partial\bar\partial(\varrho_0\cdot
f_0(\bbr^2)\cdot i\partial\bar\partial
f_0(\bbr^2))\bigr)+i\partial\bar\partial\bigl(\varrho_t\cdot f_
t(\bbr^2)\cdot i\partial\bar\partial f_t(\bbr^2)\bigr).
\end{equation}
It is well-defined and is a $d$-closed $4$-form on $X_t$. It decomposes
$$\Theta_t=\Theta_t^{3,1}+\Theta_t^{2,2}+\Theta_t^{1,3}.
$$

We claim that for $t$ sufficiently small, $\Theta_t^{2,2}$ is
positive definite. Indeed, over $V_t(\frac{1}{2})$, the first term
in \eqref{Phi} is trivial and, since  $\rho_t=1$,
$$\Theta_t^{2,2}|_{V_t(\frac 1
2)}=\Theta_t|_{V_t(\frac 1 2)}=\omega^2_{\co,t}>0.
$$
Over {$X_t[\ha]$}, we argue that
\begin{equation}\lab{lim-2}
\lim_{t\to 0}\Theta_t|_{{X_t[\ha]}}=\Omega_0|_{{X_0[\ha]}}
\end{equation}
uniformly. From the expression of $\Theta_t$ it is clear that
$\Theta_t$ only involves $f_u(s)$ and its derivatives up to second
order. Hence by {(1) of} the previous Lemma, we see that over
$V_t(1)\setminus V_t(\frac 1 2)$, $f_t(\brr)$ and its partial
derivatives up to second order all converges uniformly to that of
$f_0(\brr)$. Hence since {$X_0[\ha]$} is compact and is disjoint
from the singular points, the limit \eqref{lim-2} holds uniformly.
In the end, since the part $\Theta_ t^{3,1}$ and $\Theta_t^{1,3}$
are trivial over $V_t(\ha)$ and that the complex structure of $X_t$
varies smoothly in $t$, the part $\Theta_t^{1,3}$ and
$\Theta_t^{3,1}$ converges to zero uniformly as $t\to 0$.
Consequently, for sufficiently small $\eps$, $\Theta_t^{2,2}$ is
positive on $X_t[\ha]$ for $|t|<\eps$. Combined with the positivity
of $\Theta_t^{2,2}$ over $V_t(\ha)$, we obtained the desired
positivity of $\Theta_t^{2,2}$ for $t$ sufficiently small.

We let $\omega_t$ be the hermitian form on $X_t$ such that
$(\omega_t)^2=\Theta^{2,2}_t$. Note that for small $t$, these
metrics have uniform geometry on $X_t[\ha]$ and are Ricci-flat
K\"ahler metric over $V_t(\frac 1 2)$.
In the following we will use $\omega_t$
as our background metric on $X_t$. Therefore objects such as norms
and volume forms on $X_t$ are all taken with respect to $\omega_t$.

\va
Recall that our goal is to find balanced metrics on $X_t$. We will
achieve this by modifying the form $\Theta_t^{2,2}$ to make it both
closed and positive definite.

Since $\Theta_t$ is $d$-closed on $X_t$,
$$\bar\partial
\Theta_t^{2,2}=-\partial\Theta_t^{1,3}.
$$
We claim that for sufficiently small $t$,
$H^{1,3}(X_t,\mathbb{C})=0$.  By the Dolbeault theorm and Serre
duality, $H^{1,3}(X_t,\CC)=H^3(X_t,T^\ast_{X_t})=H^0(X_t, T_{X_t})$.
Thus $H^{1,3}(X_t,\mathbb{C})=0$ is equivalent to $H^0(T_{X_t})=0$.
We consider the total family of the deformation $\pi:\cX\to\Delta$
(cf. Definition \ref{total}). Let $\cX^*=\cX-\{p_1,\cdots,p_l\}$.
Then $\cX\sta\to\Delta$ is smooth and the relative tangent bundle
$T_{\cX\sta/\Delta}$ is a vector bundle on $\cX\sta$.

Now suppose for infinitely many $t\in\Delta$ approaching
$0\in\Delta$, $H^0(T_{X_t})\ne 0$, then either by using elliptic
estimations or applying the work of log-differential,
one concludes that $H^0(T_{X_{0,\sm}})\ne 0$, where $X_{0,\sm}=
X_0-\{p_1,\cdots,p_l\}$. Let $\bE\sub Y$ be the union of the
contracted rational curves under the projection $Y\to X_0$; then
$Y-\bE = X_{0,\sm}$, and thus $H^0(T_{X_{0,\sm}})=H^0(T_{Y-\bE})$.
On the other hand, since $Y$ is smooth and $\bE\sub Y$ is a
codimension 2 complex submanifold, by Hartogs' Lemma,
$H^0(T_Y)=H^0(T_{Y-\bE})\ne 0$.

Since $Y$ is K\"ahler and its
fundamental group is finite, we have the following vanishing results.
First, using $H^1(Y,\CC)=0$, we obtain $H^{0,1}(Y,\mathbb C)=H^{1,0}(Y,\mathbb
C)=0$. Since the canonical line bundle of $Y$ is trivial, this implies $H^{3,1}(Y,\mathbb C)=H^{1,3}(Y,\mathbb C)=0$.
Applying the Serre duality, we get $H^0(T_Y)=H^{1,3}(Y,\CC)=0$, contradicting to $H^0(T_Y)\ne 0$
stated earlier. Thus $H^0(T_{X_t})=0$ for sufficiently small $t$.

\vsp

Therefore there is a $(1,2)$-form $\nu_t$ on $X_t$ such that $\dbar
\nu_t=-\Theta_t^{1,3}$. 
We let $\mu_t$ be a $(1,2)$-form on $X_t$ such that
\begin{equation}\lab{301}
i\partial\bar\partial\mu_t=-\partial\Theta_ t^{1,3}=\bar\partial
\Theta_t^{2,2} \and \mu_t\perp_{\omega_t} \ker\partial\bar\partial.
\end{equation}
We  define
\begin{equation}\lab{Omega}
\Omega_t=\Theta_t^{2,2}+\theta_t+\bar\theta_t,\quad
\theta_t=i\partial\mu_t.
\end{equation}
Then (\ref{301}) implies
\begin{equation*}
\bar\partial\Omega_t=\bar\partial\Theta_t^{2,2}+
\bar\partial(i\partial\mu_t)+\bar\partial(-i\bar\partial\bar\mu_t)=0,
\end{equation*}
and since $\Omega_t$ is real, $\Omega_t$ is $d$-closed.

\begin{prop}\lab{po}
For sufficiently small $ t$, $\Omega_t$ is  positive.
\end{prop}

Once this is proved, then the hermitian form $\tilde\omega_t$
defined via $(\tilde\omega_t)^2=\Omega_t$ is a balanced metric on
$X_t$.

\section{The positivity of $\Omega_t$}

To prove the Proposition, we first show that for the $C^0$-norm $\parallel\!\cdot\!\parallel_{C^0}$
measured using $\omega_t$, we have

\begin{lemm}
Suppose $\lim_{t\to0}\norml \theta_t\normr_{C^0}=0$, then $\Omega_t$
is positive for small $t$.
\end{lemm}

\begin{proof}
We let $\ast_t$ be the Hodge operator associated to the hermitian
metric $\omega_t$. Then
$$\ast_t\Omega_t=\omega_t+\ast_t(\theta_
t+\bar\theta_t)
$$
and $\Omega_t$ is positive if $\omega_t+\ast_t(\theta_
t+\bar\theta_t)$ is positive.

Now let $q_t$ be any closed point of $X_t$ and let $(z_\alpha)$ be a
local chart of $X_t$ at $q_t$ so that
$$\omega_t(q_t)=\sqrt{-1}\delta_{\alpha\beta}dz_\alpha\wedge d\bar z_\beta
\and \ast_t(\theta_
t+\bar\theta_t)(q_t)=\vartheta=\sqrt{-1}\vartheta_{\alpha\bar
\beta}dz_\alpha\wedge d\bar z_\beta.
$$
Thus $\omega_t+\ast_t(\theta_t+\bar\theta_t)$ is positive at $q_t$
if and only if the matrix $\bigl(\delta_{\alpha\beta}+\vartheta_{\alpha\bar
\beta}\bigr)_{1\leq \alpha,\beta\leq 3}$ is positive. Since
$\omega_t(q_t)=\sqrt{-1}\delta_{\alpha\beta}dz_\alpha\wedge d\bar z_\beta$,
\begin{eqnarray*}
\sum_{\alpha,\beta}| \vartheta_{\alpha\bar \beta}| ^2=| \ast_t(\theta_
t+\bar\theta_t)(q_t)| ^2=| (\theta_t+\bar\theta_t)(q_t)| ^2 \leq 4|
\theta(q_t)| ^2.
\end{eqnarray*}
Thus if $| \theta(q_t)|^2$ is small, the matrix $\bigl(\delta_{\alpha
\beta}+\bar\vartheta_{\alpha\bar \beta}\bigr)$ is positive. This proves that if the
$C^0$-norm $\norml\theta_t\normr_{C^0}$ is small, the form
$\ast_t\Omega_t$, and hence the form $\Omega_t$, is positive.
\end{proof}

The estimate of $\parallel\!\!\theta_t\!\!\parallel_{C^0}$ will be
achieved in the remainder part of this section.

To estimate $\theta_t$, we use the $4^{th}$-order differential
operator $E_t$ (first introduced in \cite{KS}) on
$\Lambda^{2,3}(X_t)$:
\begin{equation*}
E_t=\partial\bar\partial\bar
\partial^\ast\partial^\ast
+\partial^\ast\bar\partial\bar
\partial^\ast\partial
+\partial^\ast\partial.
\end{equation*}
Here  the adjoint operators $\partial^\ast=-\ast
\bar\partial\ast$ and $\bar\partial^\ast=-\ast\partial\ast$ (the
same as $\bar\vartheta$ and $\vartheta$  in \cite{MK}) are defined
using the hermitian metric $\omega_t$ on $X_t$. In \cite{KS},
Kodaira and Spencer proved that $E_t$ are self-adjoint, strongly
elliptic of order $4$, and a form $\phi\in \Omega^{2,3}(X_t)$
satisfying $E_t\phi=0$ if and only if
\begin{equation}\lab{302}
\partial\phi=0 \and
\bar\partial^\ast\partial^\ast\phi=0.
\end{equation}

We now let $\gamma_t$ be a solution of
\begin{equation}\lab{303}
E_t(\gamma_t)=-\partial\Theta_t^{1,3}.
\end{equation}
We first check that $-\partial\Theta_t^{1,3}\bot \ker E_t$. Let
$\phi\in\ker E_t$, from (\ref{302}) we have
$\bar\partial^\ast\partial^\ast \phi=0$; from (\ref{301}) we have
\begin{equation*}
(-\partial\Theta_t^{1,3},\phi)=(i\partial\bar\partial
\mu_t,\phi)=(i\mu_t,\bar\partial^\ast\partial^\ast\phi\bigr)=0.
\end{equation*}
This implies $-\partial\Theta_t^{1,3}\bot\ker E_t$. By the theory of
elliptic operators, there is a unique smooth solution $\gamma_ t\bot
\ker E_t$ of (\ref{303}).

We claim that the $\gamma_t$ and the $\mu _t$ defined in \eqref{301}
are related by
\begin{equation}\lab{304}
i\mu_t=\bar\partial^\ast\partial^\ast\gamma_t\quad \text{and}\quad
\partial\gamma_t=0.
\end{equation}
This can be seen as follows. From (\ref{301}) and (\ref{303}), we
get $E_t(\gamma_t)-i\partial\bar\partial\mu_t=0$, which, from the
definition of the operator $E_t$, is equivalent to
\begin{equation*}
\partial\bar\partial(\bar\partial^\ast
\partial^\ast\gamma_t-i\mu_t)
+\partial^\ast(\bar\partial\bar\partial^\ast+1)
\partial\gamma_t=0.
\end{equation*}
By taking  the $L^2$-norm of the left hand side, we get
\begin{equation}\lab{305}
\partial\bar\partial(\bar\partial^\ast
\partial^\ast\gamma_t-i\mu_t)
=0\ \ \ \text{and}\ \ \
\partial^\ast(\bar\partial\bar
\partial^\ast+1)\partial\gamma_t=0.
\end{equation}
On the other hand, for any $\phi\in\ker
\partial\bar\partial$, we have
$(\bar\partial^\ast\partial^\ast\gamma_t,\phi) =(\gamma_
t,\partial\bar\partial\phi)=0$.
Since $\mu_t\bot\ker\partial\bar\partial$,
\begin{equation}\lab{306}
(\bar\partial^\ast\partial^\ast\gamma_t-i\mu_t)\bot \ker
\partial\bar\partial.
\end{equation}
Combining (\ref{305}) and (\ref{306}), we obtain
$\bar\partial^\ast\partial^\ast\gamma_t-i\mu_t=0$, which is the
first identity in (\ref{304}). The second in (\ref{304}) follows
from the second equality of (\ref{305}), since
\begin{equation*} 0=\int_{X_t}\langle \partial^\ast(\bar\partial\bar\partial^\ast+1)
\partial\gamma_
t,\gamma_t\rangle=\int_{X_t}(| \bar\partial^\ast\partial\gamma_t|
^2+|
\partial\gamma_t| ^2).
\end{equation*}

Summarizing

\begin{lemm}\lab{eq-10}
We let $\gamma_t$ be the unique solution to
$E_t(\gamma_t)=-\partial\Theta_t^{1,3}$ subject to the constraint
$\gamma_t\bot \ker E_t$. Then $\gamma_t$ satisfies
$\partial\gamma_t=0$ and the $\theta_t$ defined in \eqref{Omega} is
of the form
{$\theta_t=\partial\bar\partial\sta\partial\sta\gamma_t$.}
\end{lemm}

Because of this Lemma, we can apply elliptic estimate to bound the
norm of $\gamma_t$ by that of $\partial\Theta_t^{1,3}$. We first
check that for any given $0<c<\frac 1 2$, $E_t$ converges uniformly
to the $E_0$ on $X_0[c]$. Since $E_t$ depends on the complex
structure of $X_t$ and the Hodge star operator of the background
metric $\omega_t$, it depends on the derivatives of the components
of $\omega_t$ of order at most four. By Lemma \ref{control} and the
discussion following the Lemma, for $c>0$, over $X_t[c]$ the
Hermition forms $\omega_t$ converges to $\omega_0$ in $C^4$. Then
because for any $0<c<\frac 1 2$ and $t$ sufficiently small, the
Riemannian manifolds with boundaries $(X _t[c], \omega_t)$ have
uniform geometry, there is a constant $C$ independent of $t$ small
so that
\begin{equation}\lab{307}
\norml\gamma_t\normr_{L^2_4\bl X_t[2c]\br} \leq
C\Bigl(\norml\gamma_t\normr_{L^2\bigl(X_
t[c]\bigr)}+\norml\partial\Theta_t^{1,3}\normr_{L^2\bigl(X_t[c]\bigr)}\Bigr).
\end{equation}

To proceed, we argue that the quantity $\int_{ X_t}|
\partial\Theta^{1,3}_t| ^2$ is bounded by $C|t|^2$ for some constant
$C$.  Indeed,  using the explicit expression
\eqref{Phi}, $\Theta_t^{1,3}$ has compact support contained in
$X_t[\frac 1 2]$ and only depend on $\psi_t$ and the complex
structure of $X_t$. Because $\psi_t$ are smooth in $t$ and the
complex structure are also smooth on $X_t[\frac 1 2]$,
$\partial\Theta_t^{1,3}$ are smooth in $t$. Then because $\Theta_0=\Omega_0$
is of type $(2,2)$ and $d$-closed, we have
$\partial\Theta_0^{1,3}=(\partial\Theta_0)^{2,3}=0$ and therefore
\begin{equation}\lab{310}
\sup_{z\in X_t} \norml \partial\Theta_t^{1,3}(z)\normr<C| t|.
\end{equation}
This provides a bound on the last term in the inequality
\eqref{307}.


\va

Proposition \ref{po} will follow from the following stronger estimate.

\begin{prop}\lab{7}
For any $ \kappa>-\frac  4 3$,
\begin{equation*}
  \lim_{t\to
0}\bigl(| t|^{\kappa}\sup_{X_t}| \theta_t|
^2_{\omega_t}\bigr)=0.
\end{equation*}
\end{prop}

\begin{proof}
First, according to the Sobolev imbedding theorem, since $X_t[\frac
1 8]$ have uniform geometry and $E_t$ converges uniformly to $E_0$
on $X_0[\frac 1 8]$,  there is a constant $C$ independent of $t$ so
that for $p>6$,
\begin{equation*}
\parallel\!\!\gamma_t\!\!\parallel_{C^3(X_t[1/4])}\leq
C\bl\parallel\!\!\gamma_t\!\!\parallel_{L^2(X_t[1/8])}+\parallel\!\!\partial\Theta_t^{1,3}\!\!\parallel_{L^p(X_t[1/8])}\br.
\end{equation*}
Because of the identities in Lemma \ref{eq-10} and the inequality
(\ref{310}), there is a constant $C$ independent of $t$ so that
\begin{equation*}
\ \ \ \sup_{X_t[\frac 1 4]}| \theta_t| ^2\leq C\bl
|t|^2+\int_{X_t[\frac 1 8]} | \gamma_t|^2\br.
\end{equation*}
Multiplying by $| t|^{\kappa}$ on both sides, we get
\begin{equation}\lab{312}
\lim_{t\to   0}\bigl(| t|^{\kappa}\!\sup_{ X_t[\frac 1 4]}|
\theta_t| ^2\bigr)\leq C\lim_{t\to 0}| t|^{\kappa}\int_{X_t}|
\gamma_t| ^2.
\end{equation}
This provides us the bound we need over $X_t[\frac{1}{4}]$.

To control that over its complement, namely that inside $V_t(\frac{1}{4})$, we need the following two
Lemmas whose proofs will be postponed until the next section.

\begin{lemm}\lab{313}
There is a constant $C$ independent  of $t$ such that
\begin{eqnarray*}
\sup_{V_t(\frac 1 4)}| \theta_t| ^2\leq  C\int_{V_ t(\frac 1 4)}|
\theta_t| ^2\bbr^{-4}+C\sup_{X_t[\frac 1 4]}| \theta_t| ^2.
\end{eqnarray*}
\end{lemm}

\begin{lemm}\lab{HK}
There is a constant $C$ independent of $t$ such that
\begin{equation*}
\int_{V_t(\frac 1 4)}| \theta_t| ^2\bbr^{-\frac 8 3}\leq
C\int_{X_t[\frac 1 8]}(| \gamma_t| ^2+|
\partial\Theta^{1,3}_t| ^2).
\end{equation*}
\end{lemm}



We continue the proof of Proposition \ref{7}. Until the end of this
section, all constant $C$'s are independent of $t$; also when it
depends on some other data, like a choice of $\delta>0$, we shall
use $C(\delta)$ to indicate so.

Since $\bbr^2\geq |t|$ over $V_t(1)$, Lemma \ref{HK}
implies
\begin{eqnarray*}
\begin{aligned}
\int_{V_t(\frac 1 4)}| \theta_t| ^2\bbr^{-4} \leq C_1|
t|^{-\frac 2 3}\int_{X_t[\frac 1 8]}(| \gamma_t| ^2+|
\partial\Theta^{1,3}_t| ^2).
\end{aligned}
\end{eqnarray*}
Combined with Lemma \ref{313}, we have
\begin{eqnarray*}
\sup_{V_t(\frac 1 4)}| \theta_t| ^2\leq C_2|t| ^{-\frac 2
3}\int_{X_t[\frac 1 8]}(|\gamma_t|^2+|\partial\Theta^{1,3}_t|
^2)+C\sup_{X_t[\frac 1 4]}| \theta_t| ^2.
\end{eqnarray*}
Then multiplying $|t|^{\kappa}$ to both sides and taking limit
$t\to 0$, we find that by (\ref{310}) the second term on the right hand
vanishes since $-\frac 2 3+\kappa>-2$, and the third one can be
controlled by the first one  in view of (\ref{312}). So we get
\begin{equation*}\lab{011308}
\lim_{t\to    0}\Bigl(| t|^{\kappa}\sup_{V_t(\frac 1 4)}| \theta_t|
^2\Bigr) \leq C_3\lim_{t\to    0}| t|^{-\frac 2 3+\kappa}\int_{X_t}|
\gamma_t|^2.
\end{equation*}
Therefore by (\ref{312}) and the above inequality, should Proposition
\ref{7} fail we must have
\begin{equation*}
\overline{\lim}_{t\to    0}| t|^{-\frac 2 3+\kappa}\int_{X_t}|
\gamma_t| ^2>0.
\end{equation*}
In this case, there is a positive $\alpha>0$ and a sequence $t_i\to
0$ such that
\begin{equation*}
| t_i|^{-\frac 2 3+\kappa}\int_{X_{t_i}}| \gamma_{t_i}|
^2=\alpha_{i}^2\geq \alpha^2.
\end{equation*}
Normalizing
$\tilde\gamma_{t_i}=t_i^{-\frac 1 3+\frac \kappa
2}\alpha_{i}^{-1}\gamma_{t_i}$,
it satisfies
\begin{equation}\lab{314}
E_{t_i}(\tilde\gamma_{t_i})=-t_i^{-\frac 1 3 +\frac \kappa
2}\alpha_{i}^{-1}\partial\Theta_{t_i}^{1,3}
\end{equation}
and
\begin{equation}\lab{3014}
\int_{X_{t_i}}| \tilde\gamma_{t_i}| ^2=1.
\end{equation}
Since $-\frac 1 3+\frac \kappa 2>-1$, \eqref{310} implies that the
$C^0$-norm of the right hand side of (\ref{314}) uniformly goes to
zero when $i\to\infty$. Therefore by passing to a subsequence,
there exists a smooth $(1,3)$-form $\tilde\gamma_0$ on
$X_{0,\sm}$\footnote{$X_{0,\sm}$ is the smooth loci of $X_0$.} such
that $E_0(\tilde\gamma_0)=0$ and $\tilde\gamma_{t_i}\to\tilde
\gamma_0$ pointwise.

To make sure that the limit is non-trivial, we check that
$\norml \tilde\gamma_0\normr_{L^2}>0$.
For this, we need the following estimate that will be proved in the next section.

\begin{lemm}\lab{est}
For any $0<\iota<\frac 1 3$, there is a constant $C$ such that for
any $0<\delta<\frac 1 4$ and $|t|<\delta$,
\begin{equation*}
\int_{V_t(\delta)}| \gamma_t| ^2\bbr^{-\frac 4 3}\leq
C\delta^{2\iota}\int_{X_t[\frac 1 8]}(| \gamma_t| ^2+|
\partial\Theta_t^{1,3}| ^2).
\end{equation*}
\end{lemm}

We continue our proof of $\norml \tilde\gamma_0\normr_{L^2}>0$. By
(\ref{310}) and (\ref{3014}), for large $i$
\begin{equation}\lab{315}
\int_{V_ {t_i}(\delta)}| \tilde\gamma_{t_i}| ^2\bbr^{-\frac 4 3}\leq
C_4\delta^{2\iota}\int_{X_ {t_i}[\frac 1 8]}(| \tilde\gamma_{t_i}|
^2+ {t_i^{-\frac 2 3+\kappa}\alpha_i^{-2}}{|
\partial_{t_i}\Theta_{t_i}^{1,3}| ^2})\leq C_5\delta^{2\iota}.
\end{equation}
Letting $i\to\infty$ and using Lemma \ref{control}(2),
we get
\begin{equation}\lab{3151}
\int_{V_0(\delta)\sta}| \tilde \gamma_{0}| ^2\bbr^{-\frac 4
3}\leq C_5\delta^{2\iota},
\end{equation}
where $V_0(\delta)\sta=V_0(\delta)\setminus\{p_1,\cdots, p_l\}$.
Because of (\ref{3014}) and the pointwise convergence
$\tilde\gamma_{t_i}\to\tilde\gamma_0$ over $X_{0,\sm}$, we have
\begin{equation*}
\int_{X_{0,\sm}}| \tilde\gamma_0| ^2\geq 1-C_5\delta^{2\iota};
\end{equation*}
since $\delta$ is arbitrary, we obtain
\begin{equation}\lab{316}
\int_{X_{0,\sm}}| \tilde\gamma_0| ^2= 1.
\end{equation}

To obtain a contradiction to complete the proof of Proposition \ref{7}, we now show that
$\tilde\gamma_0=0$. We first show that
$\partial\sta\tilde\gamma_0=0$. Since $\partial\gamma_t=0$,
$$E_t(\gamma_t)=\partial\dbar\dbar^\ast\partial^\ast\gamma_t;
$$
consequently,
\begin{equation*}
\int_{X_t}| \bar\partial^\ast\partial^\ast\gamma_t|
^2=\int_{X_t}\langle E_t(\gamma_t),\gamma_t\rangle.
\end{equation*}
Substituting $\tilde\gamma_{t_i}$ and applying the H\"older
inequality, (\ref{3014}), (\ref{314}) and (\ref{310}), we obtain
\begin{eqnarray*}
\int_{X_{t_i}}| \bar\partial^\ast\partial^\ast\tilde\gamma_{t_i}| ^2
\leq \Bigl(\int_{X_{t_i}}| \tilde\gamma_{t_i}| ^2\Bigr)^{\frac 1
2}\Bigl(\int_{X_{t_i}}| E_{t_i}(\tilde\gamma_{t_i})| ^2\Bigr)^{\frac
1 2}\leq C_6|t_i| ^{\frac 2 3+\frac \kappa 2}.
\end{eqnarray*}
Taking $i\to\infty$ and noticing $\kappa>-\frac 4 3$, we get
$\dbar^\ast\partial^\ast\tilde\gamma_0=0$.

We next pick a cut-off function $\tau(s)$ that vanishes when
$s\leq 0$ and $\tau(s)=1$ when $s\geq 1$. For any $0<\delta<1$, we
put ${\bf s}_\delta=\frac {2\bbr-\delta}{\delta}$. (Note that
$\bbr$ is a function on $V_0(1)$ defined in (\ref{bbr}) and is equal
to $r\circ \xi_i^{-1}$.) We define
$$\tau_\delta=\tau(\bs_\delta).
$$
It vanishes in a small neighborhood of $\{p_1,\cdots,p_l\}\sub X_0$;
it takes value $1$ near the boundary of
$V_0(1)$. We then extend to a function on $X_0$ by assigning value $1$
elsewhere. We
still denote this extension by $\tau_\delta$.

Using (\ref{202}) and
(\ref{203}), over $V_0(\delta)\setminus V_0(\frac\delta 2)$ and for
a constant $C_7$ independent of $\delta$, we have
\begin{equation}\lab{322}
| \partial\tau_\delta| _{\omega_{\co,0}}^2=\frac 4 {\delta^2}|
\tau'(s)| ^2| \partial \bbr| _{\omega_{\co,0}}^2\leq C_7\bbr^{-\frac 4
3}.
\end{equation}

We now fix a $\delta_1<\frac 1 8$. Since
$\tau_{\delta_1}\partial^\ast\tilde\gamma_0$ has compact support, we
can view $\tau_{\delta_1}\partial^\ast\tilde\gamma_0$ as a
$(1,3)$-form on $Y$. Since
$H^{1,3}(Y,\mathbb{C})=0$, (cf. discussion preceding \eqref{301}), there exists a smooth $(1,2)$-form
$\varsigma_{\delta_1}$ on $Y$ such that
$$\tau_{\delta_1}\partial^\ast\tilde\gamma_0=\dbar\varsigma_{\delta_1}.
$$
Then for any $\delta<\frac {\delta_1} 2$, by integration by parts
and using $\dbar^\ast\partial^\ast\tilde\gamma_0=0$,
\begin{equation}\lab{323}
\int_{X_0}\tau_{\delta_1}| \partial^\ast\tilde\gamma_0| ^2
=\int_{X_0}\tau_\delta\tau_{\delta_1}|
\partial^\ast\tilde\gamma_0| ^2 =\int_{X_0}\tau_\delta\langle
\partial^\ast\tilde\gamma_0,\dbar\varsigma_{\delta_1}\rangle
=
\int_{X_0}
\langle\ast(\partial\tau_\delta\wedge\ast\partial^\ast\tilde\gamma_0),\varsigma_{\delta_1}\rangle.
\end{equation}
By the H\"older inequality and the definition of $\tau_\delta$, the
right hand side obeys
\begin{equation}\lab{324}
\int_{X_0}\langle\ast(\partial\tau_\delta\wedge\ast\partial^\ast
\tilde\gamma_0),\varsigma_{\delta_1}\rangle \leq
\Big(\int_{V(\delta)\setminus V(\frac \delta 2)}|
\partial\tau_\delta| ^2| \partial^\ast\tilde\gamma_0| ^2\Bigr)^{\frac
1 2}\Bigl(\int_{V(\delta)\setminus V(\frac\delta 2)}|
\varsigma_{\delta_1}| ^2\Bigr)^{\frac 1 2}.
\end{equation}
We then apply the following estimate to be proved in the next section:

\begin{lemm}\lab{est2} For any $0<\iota<\frac 1 3$, there is a constant $C$ such that for
any $0<\delta<\frac 1 4$ and any $|t|<\delta$,
\begin{equation*}
\int_{V_t(\delta)}| \partial^\ast\gamma_t| ^2\bbr^{-\frac 4
3}<C\delta^{2\iota}\int_{X_t[\frac 1 8]}(| \gamma_t| ^2+|
\partial\Theta_t^{1,3}| ^2).
\end{equation*}
\end{lemm}

From this Lemma, (\ref{310}) and (\ref{3014}), we obtain for large
$i$,
\begin{equation*}
\int_{V_ {t_i}(\delta)}| \partial^\ast\tilde\gamma_{t_i}|
^2\bbr^{-\frac 4 3}<C_8\delta^{2\iota}\int_{X_ {t_i}[\frac 1
8]}(|\tilde\gamma_{t_i}|^2+\alpha_i^{-2}| t_i| ^{-\frac 2 3+\kappa}|
\partial\Theta_{t_i}^{1,3}|^2)\leq C_8\delta^{2\iota},
\end{equation*}
where $C_8$ is independent of $\delta$. Taking limit $i\to
\infty$ and using Lemma \ref{control}(2), we get
\begin{equation*}
\int_{ V_0(\delta)\sta}|
\partial^\ast\tilde\gamma_0| ^2\bbr^{-\frac 4 3}\leq
C_8\delta^{2\iota}.
\end{equation*}
This inequality and (\ref{322}) imply
\begin{equation}\lab{319}
\int_{V_0(\delta)\setminus V_0(\frac \delta 2)}|
\partial\tau_\delta| ^2| \partial^\ast\tilde\gamma_0| ^2 \leq
C_9\delta^{2\iota}.
\end{equation}

Next, we denote by $\bU(\delta)$ the union of all neighborhoods
$U_i(\delta)$ of $E_i$ in $Y$, defined in the previous section for
$0<\delta<1$. Over $V_0(1)^\ast=V_0(1)-\{p_1,\cdots,p_l\}\cong \bU(1)\setminus \bE$
we have three metrics:
$$\omega_e=i\partial\bar\partial \bbr^2,\quad
\omega_{\co,0}=i\frac 3 2
\partial\bar\partial(\bbr^2)^{\frac 2 3}\and
\omega_{\co}.
$$
(Recall that $\omega_{\co,0}$ is the cone Ricci-flat metric and
$\omega_{\co}$ is the Ricci-flat metric on $\bU(1)$ (see
\eqref{cos}). Via isomorphism $\xi$,
$\xi^\ast(\omega_e)=i\partial\bar\partial r^2$ is a metric induced
from the Euclidean metric on $\mathbb C^4$.) Since all these metrics
are homogeneous, to compare them we only need to compare their
restrictions over a single point in one $E_i$, say at $z=0$.

We now compare the metrics $\omega_{\co,0}$ and $\omega_{e}$ by (\ref{203})
and (\ref{201}); compare the metrics $\omega_e$ and
$\omega_{\co}$ by (\ref{201}) and (\ref{209}). Since
$\varsigma_{\delta_1}$ is a $(1,2)$-form, the second factor in
(\ref{324}) fits into the inequalities
\begin{equation}\lab{le}
\int_{V_0(\delta)\setminus V_0(\frac\delta 2)}|
\varsigma_{\delta_1}| ^2 \leq  C \int_{V_0(\delta)\setminus
V_0(\frac \delta 2)}| \varsigma_{\delta_1}|
^2_{\omega_e}\text{vol}_{\omega_e} \leq
C_{10}\int_{V_0(\delta)\setminus V_0(\frac\delta 2)}|
\varsigma_{\delta_1}|
^2_{\omega_{\co}}\bbr^{-4}\text{vol}_{\omega_{e}}.
\end{equation}
Since $\varsigma_{\delta_1}$ and $\omega_{\co}$ are
smooth on $\bU(\delta_1)$, there exists a constant
$C_{11}(\delta_1)$, possibly depending on $\delta_1$, such that
$$\max_{\bU(\delta_1)}| \varsigma_{\delta_1}| ^2_{\omega_{\co}}\leq
C_{11}(\delta_1).
$$
Therefore
\begin{equation*}
\int_{V_0(\delta)\setminus V_0(\frac\delta 2)}|
\varsigma_{\delta_1}|^2_{\omega_{\co}}\bbr^{-4}\text{vol}_{\omega_e}\leq
C_{11}(\delta_1)\int_{\{\bbr=1\}}\int_{\frac \delta 2}^\delta \bbr
d\bbr dS\leq C_{12}(\delta_1)\delta^2,
\end{equation*}
where $dS$ is the volume element of the surface $\{\bbr=1\}$. Combined
with (\ref{le}), we get
\begin{equation*}
\int_{V_0(\delta)\setminus V_0(\frac\delta 2)}|
\varsigma_{\delta_1}| ^2 \leq C_{13}(\delta_1)\delta^2.
\end{equation*}
Then combined with (\ref{319}) and (\ref{324}), we obtain
\begin{equation*}
\int_{X_0}\langle\ast(\partial\tau_\delta\wedge\ast
\partial^\ast\tilde\gamma_0),\varsigma_{\delta_1}\rangle\leq
C_{14}(\delta_1)\delta^{1+\iota},
\end{equation*}
and with  (\ref{323}),
\begin{equation*}
\int_{X_0}\tau_{\delta_1}| \partial^\ast\tilde\gamma_0| ^2\leq
C_{15}(\delta_1)\delta^{1+\iota}.
\end{equation*}
Taking $\delta\to 0$ and then $\delta_1\to 0$, we get
$\int_{X_{0,\sm}}| \partial^\ast\tilde\gamma_0| ^2=0$; hence
$\partial^\ast\tilde\gamma_0=0$.
\\

It remains to show that $\tilde\gamma_0=0$. Since
$\partial\gamma_t=0$, we have $\partial\tilde\gamma_0=0$.
Let $\varphi_0\defeq \bar{\tilde\gamma}_0|_{V_0(\frac 1 4)\sta}$ be the complex conjugate. Then
$\bar\partial\varphi_0=\dbar^\ast\varphi_0=0$. On the other hand,
comparing the metrics (\ref{201}) and (\ref{203}), and using
(\ref{3151}),  we have
\begin{equation*}
\int_{V_0(\frac 1 4)\sta}| \varphi_0|
_{\omega_e}^2\text{vol}_{\omega_e}\leq C\int_{V_0(\frac 1 4)\sta}|
\varphi_0| ^2\bbr^{-\frac 4 3}<+\infty.
\end{equation*}
Therefore, $\varphi_0\in H^{3,2}_{(2)}\big(V_0(\frac 1
4)\sta,\omega_e\bigr)$ is an $L^2$-Dolbeault cohomology class of
$V_0(\frac{1}{4})\sta$, with respect to $\omega_e$.

We claim that this cohomology group vanishes. First, for any
$0<\delta<\frac 1 4$,
$V_0(\delta)\sta=V_0(\delta)\setminus\{p_1,\cdots, p_l\}$. If we let
$\tilde V_0(\delta)=\xi^{-1}(V_0(\delta))$, then $\tilde
V_0(\delta)$ is a disjoint union of $l$ copies of $\tilde
U_0(\delta)$, 
$$\tilde U_0(\delta)=\Bigl\{(w_1,\cdots, w_4)\in \mathbb C^4\mid
w_1^2+\cdots+w_4^2=0,r<\delta\Bigr\}.
$$
Let $\tilde\omega_e=i\partial\bar\partial r^2$ on $\tilde
U_0(\delta)^\ast=\tilde U_0(\delta)\setminus\{0\}$ be the metric
induced by the flat metric on $\CC^4$. From \cite{Oh}, we have
$\lim_{\delta\to 0}H^{3,2}_{(2)}\big(\tilde
U_0(\delta)\sta,\tilde\omega_e\bigr)=0$. Since
$\omega_e=\xi^\ast(\tilde\omega_e)$ via the isomorphism $\xi$ and
since $V_0(\delta)^\ast$ is a disjoint union of $l$ connected open
sets each of which is isomorphic to $\tilde U_0(\delta)^\ast$, we
also have $\lim_{\delta\to
0}H^{3,2}_{(2)}\big(V_0(\delta)\sta,\omega_e\bigr)=0$. Therefore,
 there exists a
$\delta_2<\frac 1 4$ and a $(3,1)$-form $\nu_0$ on
$V_0(\delta_2)^\ast$ such that $\dbar \nu_0=\varphi_0$ and
\begin{equation}\lab{327}
\int_{V_0(\delta_2)^\ast}|
\nu_0|^2_{\omega_e}\text{vol}_{\omega_e}<+\infty.
\end{equation}

Let
\begin{equation*}
\varphi_{\delta_2}=\varphi_0-\dbar\bigl((1-\tau_{\delta_2})\nu_0\bigr).
\end{equation*}
Then $\varphi_{\delta_2}$ has compact support in $X_{0,\sm}$ and
$\dbar\varphi_{\delta_2}=0$. By extension by $0$, we view
$\varphi_{\delta_2}$ as a $(3,2)$-form on $Y$. Since
$H^{3,2}(Y,\mathbb{C})=0$, which follows from $H^{0,1}(Y,\CC)=0$ and the Serre duality,
there exists a smooth function $\nu_{\delta_2}$ on $Y$
such that $\varphi_{\delta_2}=\dbar \nu_{\delta_2}$.

Now for any $\delta<\delta_2$, since $\dbar^\ast\varphi_0=0$,
\begin{equation}\lab{328}
\begin{aligned}
\int_{X_{0,\sm}}\tau_\delta| \varphi_0| ^2=&\int_{X_{0,\sm}}
\tau_\delta\langle\varphi_0,\varphi_0-\dbar((1-\tau_{\delta_2})\nu_0)
+\dbar((1-\tau_{\delta_2})\nu_0)\rangle\\
=&\int_{X_{0,\sm}}\tau_{\delta}\langle\varphi_0,\dbar(\nu_{\delta_2}+(1-\tau_{\delta_2})\nu_0)\rangle\\
=&\int_{X_{0,\sm}}
\langle\ast(\partial\tau_\delta\wedge\ast\varphi_0),\nu_{\delta_2}+(1-\tau_{\delta_2})\nu_0\rangle\\
\leq& C \Bigl(\int_{V_0(\delta)\setminus V_0(\frac \delta 2)} |
\partial\tau_\delta| ^2| \varphi_0| ^2\Bigr)^{\frac 1
2}\Bigl(\int_{V_0(\delta)\setminus V_0(\frac \delta 2 )}|
\nu_{\delta_2}| ^2+|\nu_0| ^2\Bigr)^{\frac 1 2}.
\end{aligned}
\end{equation}
Applying (\ref{3151}) and (\ref{322}), and adding
$\varphi_0=\bar{\tilde\gamma}_0|_{V_0(\frac1 4)\sta}$, we obtain
\begin{equation*}
\int_{V_0(\delta)\setminus V_0(\frac \delta 2)}|
\partial\tau_\delta| ^2| \varphi_0| ^2\leq C_{16}\delta^{2\iota},
\end{equation*}
where $C_{16}$ is independent of $\delta$. On the other hand, since
$\nu_{\delta_2}$ is a smooth form in $Y$ and $\omega_{\co}$ is a
smooth metric on $\bU(\delta_2)$, there exists a constants
$C_{17}(\delta_2)$, possibly depending on $\delta_2$, such that
$\max_{\bU (\delta_2)}|\nu_{\delta_2}|^2_{\omega_{\co}}\leq
C_{17}(\delta_2)$. This, (\ref{201}), (\ref{203}) and (\ref{209})
imply that
\begin{equation*}
\int_{V_0(\delta)\setminus V_0(\frac \delta 2)}| \nu_{\delta_2}|
^2\leq C_{18}\int_{V_0(\delta)\setminus V_0(\frac \delta
2)}|\nu_{\delta_2}|^2_{\omega_{\co}}\bbr ^{-\frac
{10}{3}}\text{vol}_{\omega_e}\leq C_{19}(\delta_2)\delta^{\frac 7
3}.
\end{equation*}
Applying (\ref{327}), we get
\begin{equation*}
\int_{V_0(\delta)\setminus V_0(\frac \delta 2)}|\nu_0|^2\leq
C\int_{V_0(\delta)\setminus V_0(\frac \delta
2)}|\nu_0|^2_{\omega_e}\bbr^{\frac 2 3}\text{vol}_{\omega_e}\leq
C_{20}\delta^{\frac 2 3}.
\end{equation*}

Substituting the above three inequalities  into (\ref{328}),  we get
\begin{equation*}
\int_{X_{0,\sm}}\tau_\delta| \varphi_0| ^2\leq
C_{21}(\delta_2)\delta^{\iota+\frac 1 3}.
\end{equation*}
Taking $\delta\to0$, since
$0< \iota< \frac 1 3$ and we have $\delta_2$ fixed, we obtain
$\int_{X_{0,\sm}}| \varphi_0| ^2=0$. This proves $\tilde\gamma_0=0$,
a contradiction that proves Proposition \ref{7}, and hence
Proposition \ref{po}.
\end{proof}

\section{Proofs of Lemmas \ref{313} to \ref{est2}}
\def\ti{\tilde}
\def\half{\frac{1}{2}}

We keep the notations introduced in the previous section. Among other
things, we have the subsets $\tilde U_i\sub \CC^4$, the
biholomorphic maps $\xi_i:\tilde U_i\to \mathcal U_i=\xi_i(\ti
U_i)\sub \cX$ and $V=\cup_{i=1}^l \mathcal U_i\sub\cX$. Using the
fiber $X_t$ of $\cX$ over $t\in\Delta$, we have biholomorphisms
$\xi_{i,t}: \ti U_{i,t}\to U_{i,t}\sub X_t$ and $V_t=\cup_{i=1}^l
U_{i,t}\sub X_t$.

Looking at the statements of Lemmas \ref{313} to \ref{est2}, they are of the form that terms of the form
$\int_{V_t(\half)}\cdot$ are bounded by a constant multiple of terms of the form $\int_{X_t[c]}\cdot$.
Since $V_t(\half)$
is a disjoint union of $l$ copies of $U_{i,t}(\half)$, by increase the multiple by $l$-fold, the Lemmas are
consequence of similar statement with $V_t(\half)$ replaced by $U_{i,t}(\half)$.

But then since all geometry of $U_{i,t}$ is alike, we only need to
prove the case where $V_t(\half)$ is replace by $U_{1,t}(\half)$.
For notational simplicity, we use $\ti U_t$ and $U_t$ to denote $\ti
U_{1,t}$ and $U_{1,t}$, respectively.


Over $\ti U_t$ ($=\ti U_{1,t}$) we have the CO-metric $\tilde{\omega}_{\co,t}\triangleq i\partial\bar\partial f_t(r^2)$,
where $f_t(s)$ is defined in (\ref{pot}). The CO-metric $\omega_{\co,t}$ on $U_{t}$ is such that
$\xi_{t}\sta(\omega_{\co,t})=\tilde\omega_{\co,t}$ (we use $\xi_t$ to denote $\xi_{1,t}=\xi|_{U_{1,t}}$);
the CO-metric on $V_t$ is $\omega_{\co,t}$ on each
$U_{t}$. The metrics $\omega_t$ on $X_t$ are deformation of $\omega_0$ away from the
singularities of $X_0$, and coincide with $\omega_{\co,t}$ over $U_t(\half)$.

One property of $\tilde \omega_{\co,t}$ we need is the following.
For  any $c$ such that $|t|^{\frac 1 2}<c<1$, the surface
$\{r=c\}\sub \tilde U_t$ is diffeomorphic to $S^2\times S^3$ and
$q=(\frac{\sqrt{c^2- t}}{\sqrt{2}},\frac{i\sqrt{c^2-
t}}{\sqrt{2}},0, t^{\frac 1 2})$ lies in this surface.
In the appendix, we prove that we can find a holomorphic coordinates
$(z_1,z_2,z_3)$ at $q$ such that the CO-metric has the form
\begin{equation}\lab{0401}
\tilde\omega_{\co,t}|_{q}=i\partial\dbar
f_t(r^2)|_q=i\sum_{\alpha=1}^3dz_\alpha\wedge d\bar z_\alpha;
\end{equation}
letting $\eta_t(s)=sf_t'(s)$, we have
\begin{equation}\lab{0402}
\partial\dbar r^2|_q=(r^2)^{\frac 1
3}\Bigl(\frac{r^4}{\eta_t^3(r^2)}\Bigr)^{\frac 1 3}\Bigr(dz_1\wedge
d\bar z_1+\frac 3 2\frac{\eta_t^3(r^2)}{r^4}dz_2\wedge d\bar
z_2+dz_3\wedge d\bar z_3\Bigr)
\end{equation}
and
\begin{equation}\lab{0403}
\partial r^2\wedge \dbar r^2|_q=\frac 3
2(r^2)^{\frac 4 3}\Bigl(\frac{\eta_t^3(r^2)}{r^4}\Bigr)^{\frac 2
3}\Bigl(1-\frac {t^2}{r^4}\Bigr)dz_2\wedge d\bar z_2.
\end{equation}

 In the appendix, we also prove that
$r^{-4}{\eta_t^3}{}$ is increasing over $[| t|,+\infty)$ and
\begin{equation} \lab{090705}
\lim_{r^2\to \mid t\mid}r^{-4}{\eta_t^3}=\frac 2 3,\quad
\lim_{r^2\to \infty}r^{-4}{\eta_t^3}=1.
\end{equation}

We let $R_{\alpha\bar \beta\gamma\bar \delta}$ be the curvature
tensor of $\tilde\omega_{\co,t}$ at $q$ in coordinate
$(z_1,z_2,z_3)$.  Let $|R_{\alpha\bar \beta\gamma\bar \delta}|$ be
its norm measured via the metric $\omega_t$. In the appendix,  we
prove

\begin{lemm}\lab{cur}
There exists a constant $C$ independent of $t$ and $r$ such that
\begin{equation*}
|R_{\alpha\bar \beta\gamma\bar \delta}|\leq Cr^{-\frac  4 3}.
\end{equation*}
\end{lemm}


Let $\tilde
\omega_e\triangleq i\partial\bar\partial r^2$ on $\tilde U_t$ be the
metric induced by the Euclidean metric on $\mathbb C^4$. Let the norm
and volume form defined by this metric be $| \cdot|
_{\tilde\omega_e}$ and $\text{vol}_{\tilde\omega_e}$. Comparing
(\ref{0401}) with (\ref{0402}), since $\tilde\omega_{\co,t}$ and
$\tilde\omega_e$ are both homogeneous, we have the relation at any
point in $\tilde U_t$:
\begin{equation}\lab{090201}
\text{vol}_{\tilde\omega_{\co,t}}=\frac 2
3r^{-2}\text{vol}_{\tilde\omega_e},
\end{equation}
and by \eqref{0401}, \eqref{0402} and the estimate \eqref{090705}, for any smooth function $f$
on $\tilde U_t$,
\begin{equation}\lab{090202}
|\nabla f| ^2_{\tilde\omega_e}\leq Cr^{-\frac 2 3}| \nabla f|
_{\tilde\omega_{\co,t}}^2.
\end{equation}
Here $|\nabla f|^2_\omega=g^{\bar \beta\alpha}\frac{\partial
f}{\partial z_\alpha}\frac{\partial f}{\partial \bar z_\beta}$ for
$\omega$ a hermitian metric having the form $\omega=g_{\alpha\bar
\beta}dz_\alpha\wedge d\bar z_\beta$, and $(g^{\bar\beta\alpha})$ is the inverse of $(g_{\alpha\bar\beta})$, that is $\sum_{\alpha}g^{\bar\beta\alpha}g_{\alpha\bar\gamma}=\delta^{\bar\beta}_{\bar\gamma}$.

We comment that the prior discussion applies to metrics
$\omega_{\co,t}$ on $U_t(\frac 1 2)$ since our chosen background
metric $\omega_t$ restricted to $U_t(\ha)$ is the CO-metric
$\tilde\omega_{\co,t}$ under the isomorphism $\xi_t$. By abuse of
notation, we shall also view $(z_1,z_2,z_3)$ as a local coordinate
of the point $\xi_t(q)\in U_t(\ha)$.

For simplicity, in the following we shall adopt the following
convention. Since we will work primarily over $X_t$, we will omit
the subscript $t$ in all the functions and forms that was used to
indicate the domain of definition. For instance, the form $\theta_t$
on $X_t$ will be abbreviated to $\theta$ when the domain manifold
$X_t$ is clear from the context. Also, we shall continue to use
$\omega_t$ to be our metric on $X_t$; thus
all norms and integrations without specification are with respect to
the metric $\omega_t$ and by the volume form of $\omega_t$. In case
we need to use a different metric, say with $\omega_e$, we will use
$|\cdot|_{\omega_e}$ and $\text{vol}_{\omega_e}$ to mean the
associated norm and volume form.

We let $\tau(\bbr)$ be a cut-off function defined on $V_t(1)$ such
that $\tau(\bbr)=1$ when $\bbr\leq \frac 1 4$ and $\tau(\bbr)=0$
when $\bbr\geq \frac 1 2$. We then extend it to $X_t$ by zero and denote by the same notation
$\tau(\bbr)$.

\begin{proof}[Proof of Lemma \ref{313}]
As commented, we only need to prove the statement of Lemma \ref{313} with
$V_t(\half)$ replaced by $U_t(\half)\defeq U_{1,t}(\half)$.
We fix a $t$ of small $|t|$. As commented, we use $\theta$ to denote
the $\theta_t$ in Lemma \ref{313}.

We introduce a sequence
$\beta_k=(\frac  3 2)^k$.
By the definition of
$\tau$ and (\ref{090201}), we have
\begin{eqnarray}\lab{090205}
\begin{aligned}
\int_{{U_t}(\frac 1 4)}| \theta| ^{2\beta_k}\bbr^{-4}\leq\frac 2
3\int_{{U_t}(\frac 1 2)}| \theta|
^{2\beta_k}\bbr^{-6}\tau^3\text{vol}_{\omega_e}.
\end{aligned}
\end{eqnarray}
The function $| \theta| ^{2\beta_k}\bbr^{-6}\tau^3$ is a non-negative
$C^\infty$-function with compact support contained in ${U_t}(\frac 1 2)$. Via
$\xi_t$, ${U_t}(\frac 1 2)$ is identified with a minimal submanifold in
$\mathbb{C}^4$ endowed with the Euclidean metric.

We quote Michael-Simon's Sobolev inequality \cite{MS} (independently
by Allard \cite{Al}): Let $M\subset\mathbb{R}^{m}$ be an
$n$-dimensional submanifold in the Euclidean $m$-space $\RR^m$, let
$H$ be its mean curvature vector, and let $f\in C^\infty_{0}(M)$ be
a nonnegative functions with compact support, then
\begin{equation*}\lab{090206}
\Bigl(\int_M
f^{\frac{n}{n-1}}\text{vol}\Bigr)^{\frac
{n-1}{n}}\leq C(n)\int_{M}(| \nabla f|+|
H|\cdot f)\text{vol}.
\end{equation*}

Applying this to the minimal submanifold $\tilde U_t\sub\CC^4$, and then applying the standard skill in page 156 of \cite{GT},
 for any nonnegative function $f$ on ${U_t}(\frac 1 2)$ with compact
support,  we see
\begin{equation*}
\Bigl(\int_{{U_t}(\frac 1 2)}f^3 \text{vol}_{\omega_e}\Bigr)^{\frac 1
3}\leq C\Bigl(\int_{{U_t}(\frac 1 2)}|\nabla
f|^2_{\omega_e}\text{vol}_{\omega_e}\Bigr)^{\frac 1 2},
\end{equation*}
where $C$ is a constant depending only on the dimension of
${U_t}(\frac 1 2)$.

Applying the volume comparison
\eqref{090201} and the norm comparison \eqref{090202} to the right hand side
of the above inequality, for $C_1$ a constant independent of $t$ we get
\begin{equation}\lab{090207}
\Bigl(\int_{{U_t}(\frac 1 2)}f^3 \text{vol}_{\omega_e}\Bigr)^{\frac 1
3}\leq C_1 \Bigl(\int_{{U_t}(\frac 1 2)}| \nabla
f|_{\omega_{\co,t}}^2\bbr^{\frac 4
3}\text{vol}_{\omega_{\co,t}}\Bigl)^{\frac 1 2}.
\end{equation}

We remark that in this section we shall use $C_i$ to denote
constants independent of $t$ and $k$. Since the exact sizes of these
constants are irrelevant, we will be very loose in keeping track of
them.

Applying (\ref{090207}) to the right hand side of (\ref{090205}) for
$f=| \theta| ^{\frac{2\beta_k}{3}}\bbr^{-2}\tau$, we have
$$
\Bigl(\int_{{U_t}(\frac{1}{4})} | \theta| ^{2\beta_k}
\bbr^{-4}\Bigr)^{\frac 2 3}   \leq \Bigl(\int_{{U_t}(\ha)}| \theta|
^{2\beta_k}\bbr^{-6}\tau^3\text{vol}_{\omega_e}\Bigr)^{\frac 2 3}
 \leq C_1^2\int_{{U_t}(\frac 1
2)}\bigl|\nabla(| \theta| ^{\frac{2\beta_k}{3}}\bbr^{-2}\tau)\bigr|^2\bbr^{\frac 4 3}
$$
\begin{equation}
\lab{090208}\qquad
\leq  3C_1^2\int_{{U_t}(\frac 1 2)}\bigl|\nabla| \theta|
^{\frac{2\beta_k}{3}}\bigr|^2\bbr^{-\frac {8}{3}}\tau^2
+3C_1^2\int_{{U_t}(\frac 1
2)}| \theta| ^{2\beta_{k-1}}| \nabla \bbr^{-2}| ^2\bbr^{\frac 4 3}\tau^2 +
\end{equation}
$$ +3C_1^2\int_{{U_t}(\frac 1 2)}| \theta|
^{2\beta_{k-1}}\bbr^{-\frac {8}{3}}| \nabla\tau| ^2.
$$

We use (\ref{0403}) to estimate the second term after the third ``$\le$'' in \eqref{090208}:
\begin{eqnarray}\lab{090210}
\begin{aligned}
\qquad \int_{{U_t}(\frac 1 2)}| \theta| ^{2\beta_{k-1}}| \nabla
\bbr^{-2}| ^2\bbr^{\frac 4 3}\tau^2 &\leq C_2 \int_{{U_t}(\frac 1
2)}| \theta| ^{2\beta_{k-1}}\bbr^{-4}\tau^2\\
&\leq C_2 \int_{{U_t}(\frac 1 4)}| \theta|
^{2\beta_{k-1}}\bbr^{-4}+4^{4}C_2\int_{X_t[\frac 1 4]}| \theta|
^{2\beta_{k-1}}.
\end{aligned}
\end{eqnarray}
From  the definition of $\tau$, the third term has an estimate:
\begin{eqnarray}\lab{090211}
\begin{aligned}
\int_{{U_t}(\frac 1 2)}| \theta| ^{2\beta_{k-1}}\bbr^{-\frac{8}{3}}|
\nabla\tau| ^2 \leq C_3 \int_{X_t[\frac 1 4]}| \theta|
^{2\beta_{k-1}}.
\end{aligned}
\end{eqnarray}

For the first term after the third ``$\le$'' in \eqref{090208}, we claim that for any $k\geq 1$,
\begin{equation}
\lab{090221}
\begin{aligned} &\int_{{U_t}(\frac 1 2)}\bigl|\nabla|
\theta| ^{\frac{2\beta_k}{3}}\bigr|^2\bbr^{-\frac 8 3}\tau^2\leq
-C_4\int_{{U_t}(\frac 1 2)}| \theta|
^{2\beta_{k-1}}\bigtriangleup_{\bar\partial}(\bbr^{-\frac 8
3}\tau^2)-
\\
&\qquad\qquad- \beta_{k-1}\int_{{U_t}(\frac 1 2)}| \theta|
^{2(\beta_{k-1}-1)}
g^{\bar\beta\alpha}(\langle\nabla_{\alpha}\nabla_{\bar\beta}\theta,\theta\rangle
+\langle\theta,\nabla_{\bar\alpha}\nabla_{\beta}\theta\rangle)
\bbr^{-\frac 8 3}\tau^2.\\
\end{aligned}
\end{equation}
Here we denote $\omega_t=\sum
g_{\alpha\bar\beta}dz_\alpha\wedge d\bar z_\beta$, (we have omitted the subscript $t$ in $g_{\alpha\bar\beta}$,) and denote the
inverse $(g_{\alpha\bar\beta})^{-1}$ of $(g_{\alpha\bar\beta})$ by
$(g^{\bar\alpha\beta})$.
We also denote $\bigtriangleup_{\bar\partial}=-g^{\bar
\beta\alpha}\frac{\partial^2}{\partial z_\alpha\partial \bar
z_\beta}$ and
$\bigtriangledown_\alpha=\bigtriangledown_{\frac{\partial}{\partial\alpha}}$
and so on.

We first prove the case $k\geq 3$. By direct calculation,
we have
\begin{equation}\lab{090213}
\int_{{U_t}(\frac 1 2)}\bigl|\nabla| \theta|
^{\frac{2\beta_k}{3}}\bigr|^2\bbr^{-\frac 8 3}\tau^2
=\frac{\beta_{k-1}^2}{4}\int_{{U_t}(\frac 1 2)}| \theta|
^{2(\beta_{k-1}-2)} \bigl|\nabla| \theta|
^2\bigr|^2\bbr^{-\frac 8 3}\tau^2.
\end{equation}
We compute
\begin{eqnarray*}
\begin{aligned}
&\beta_{k-1}(\beta_{k-1}-1)| \theta| ^{2(\beta_{k-1}-2)}\bigl|\nabla| \theta| ^2\bigr|^2\\
=&\beta_{k-1}| \theta| ^{2(\beta_{k-1}-1)}
\bigtriangleup_{\dbar}| \theta| ^2-\bigtriangleup_{\dbar}| \theta| ^{2\beta_{k-1}}\\
\leq&-\beta_{k-1}| \theta| ^{2(\beta_{k-1}-1)}g^{\bar\beta\alpha}
(\langle\nabla_\alpha\nabla_{\bar\beta}\theta,\theta\rangle
+\langle\theta,\nabla_{\bar\alpha}\nabla_\beta\theta\rangle)
-\bigtriangleup_{\dbar}| \theta| ^{2\beta_{k-1}}.
\end{aligned}
\end{eqnarray*}
Multiplying $\bbr^{-\frac  8 3}\tau^2$ to both sides of the above inequality
and then integrating over ${U_t}(\frac 1 2)$, since the CO-metric is
K\"ahler and $\tau^2$ vanishes outside ${U_t}(\frac 1 2)$,  we get
\begin{eqnarray*}
\begin{aligned}
&\beta_{k-1}(\beta_{k-1}-1)\int_{{U_t}(\frac 1 2)}| \theta|
^{2(\beta_{k-1}-2)}
\bigl|\nabla| \theta| ^2\bigr|^2\bbr^{-\frac 8 3}\tau^2\\
\leq&-\beta_{k-1}\int_{{U_t}(\frac 1 2)}| \theta|
^{2(\beta_{k-1}-1)}g^{\bar\beta\alpha}
(\langle\nabla_\alpha\nabla_{\bar\beta}\theta,\theta\rangle
+\langle\theta,\nabla_{\bar\alpha}\nabla_\beta\theta\rangle)\bbr^{-\frac
8 3}\tau^2 \\
&-\int_{{U_t}(\frac 1 2)}| \theta|
^{2\beta_{k-1}}\bigtriangleup_{\dbar}\bigl(\bbr^{-\frac 8
3}\tau^2\bigr).
\end{aligned}
\end{eqnarray*}
This and (\ref{090213}) prove \eqref{090221}.

For $k=2$, from
$
\bigtriangleup_{\dbar}| \theta| ^3=\frac 3 2 | \theta|
\bigtriangleup_{\dbar}| \theta| ^2- 3 | \theta|
\bigl|\nabla| \theta| \bigr|^2$,
a computation gives
\begin{eqnarray*}\lab{090219}
\begin{aligned}
\int_{{U_t}(\frac 1 2)}\bigl|\nabla| \theta| ^{\frac
{2\beta_2}{3}}\bigr|^2\bbr^{-\frac 8 3}\tau^2 \leq &\beta_1
\int_{{U_t}(\frac 1 2)}|  \theta| ^{2(\beta_1-1)}
\bigtriangleup_{\dbar}| \theta| ^2\bbr^{-\frac 8 3}\tau^2
-\int_{{U_t}(\frac 1 2)}| \theta|
^{2\beta_1}\bigtriangleup_{\dbar}(\bbr^{-\frac 8 3}\tau^2).
\end{aligned}
\end{eqnarray*}
This implies (\ref{090221}) in case of $k=2$.

For $k=1$, we need to estimate $\bigl|\nabla| \theta|
\bigr|^2$. When $| \theta| \not= 0$,
\begin{eqnarray*}\lab{090222}
\begin{aligned}
\bigl|\nabla| \theta| \bigr|^2=&\frac 1 4
| \theta| ^{-2} \bigl|\nabla | \theta| ^2\bigr|^2
\leq
2^{-1}g^{\bar\beta\alpha}\langle\nabla_\alpha\theta,\nabla_\beta\theta\rangle
+2^{-1}g^{\bar\beta\alpha}\langle\nabla_{\bar\beta}\theta,\nabla_{\bar\alpha}\theta\rangle\\
=&-2^{-1}\bigtriangleup_{\dbar }| \theta| ^2
-2^{-1}g^{\bar\beta\alpha}\langle\nabla_{\alpha}\nabla_{\bar\beta}\theta,\theta\rangle
-2^{-1}g^{\bar\beta\alpha}\langle\theta,\nabla_{\bar\alpha}\nabla_{\beta}\theta\rangle.
\end{aligned}
\end{eqnarray*}
When $|\theta|=0$, $|\bigtriangledown|\theta\parallel=0$ and
$-\bigtriangleup_{\bar\partial}\mid\theta\mid^2\geq 0$. We still
have above inequality. So (\ref{090221}) is valid for $k=1$ and
$\beta_0=1$.

Next we estimate the second term in (\ref{090221}) by using
Kodaira's Bochner formula (\cite{MK} p.119): for any $(p,q)$-form
$\psi=\frac{1}{p!q!}\sum \psi_{\alpha_1\cdots
\bar\beta_q}dz_{\alpha_1}\wedge\cdots\wedge d\bar z_{\beta_q}$,
\begin{eqnarray}\lab{090223}
\begin{aligned}
(\bigtriangleup_{\dbar}\psi)_{\alpha_1\cdots
\bar\beta_q}=&-\sum_{\alpha,\beta}
g^{\bar\beta\alpha}\nabla_\alpha\nabla_{\bar\beta}\psi_{\alpha_1\cdots
\bar\beta_q}\\
&+\sum_{i=1}^p\sum_{k=1}^q\sum_{\alpha,\beta}R^{\alpha\ \ \ \ \
\bar\beta}_{\ \ \alpha_i\bar\beta_k}\psi_{\alpha_1\cdots
\alpha_{i-1}\alpha\alpha_{i+1}\cdots
\bar\beta_{k-1}\bar\beta\bar\beta_{k+1}\cdots \bar\beta_q}\\
&-\sum_{k=1}^q\sum_{\beta}R_{\bar\beta_k}^{\ \
\bar\beta}\psi_{\alpha_1\cdots
\bar\beta_{k-1}\bar\beta\bar\beta_{k+1}\cdots \bar\beta_q}.
\end{aligned}
\end{eqnarray}
We use above formula to $\psi=\theta$ over ${U_t}(\frac 1 2)$.
Since $\theta=\partial\bar\partial^\ast\partial^\ast\gamma_t$,
$\partial\gamma_t=0$ and $\Theta_t^{1,3}|_{U_t(\frac 1 2)}=0$,
\begin{equation}
\bigtriangleup_{\bar\partial}\theta|_{U_t(\frac 1
2)}=(\bar\partial\bar\partial^\ast+\bar\partial^\ast\bar\partial)\theta|_{U_t(\frac
1 2)}=-\bar\partial^\ast E_t(\gamma_t)|_{U_t(\frac 1
2)}=-\bar\partial^\ast\partial\Theta^{1,3}_t |_{U_t(\frac 1 2)}=0.
\end{equation}
Then (\ref{090223}) and Lemma \ref{cur} imply
\begin{eqnarray*}\lab{090224}
\begin{aligned}
&-g^{\bar\beta\alpha}(\langle\nabla_{\alpha}\nabla_{\bar\beta}\theta,\theta\rangle
+\langle\theta,\nabla_{\bar\alpha}\nabla_{\beta}\theta\rangle)\\
= &
-g^{\bar\beta\alpha}(\langle\nabla_{\alpha}\nabla_{\bar\beta}\theta,\theta\rangle
+\langle\theta,\nabla_{\beta}\nabla_{\bar\alpha}\theta\rangle+
\langle\theta,[\nabla_{\bar\alpha},\nabla_{\beta}]\theta\rangle)
\leq C_5\bbr^{-\frac 4 3}| \theta| ^2,
\end{aligned}
\end{eqnarray*}
where $[\nabla_{\bar\alpha},
\nabla_{\beta}]=\nabla_{\bar\alpha}\nabla_\beta
-\nabla_{\beta}\nabla_{\bar\alpha}$ is the
curvature operator.

From the above inequality, we can estimate the second term after the inequality in
\eqref{090221}:
\begin{equation}\lab{090225}
\begin{aligned}
&- \beta_{k-1}\int_{{U_t}(\frac 1 2)}| \theta| ^{2(\beta_{k-1}-1)}
g^{\bar\beta\alpha}(\langle\nabla_{\alpha}\nabla_{\bar\beta}\theta,\theta\rangle
+\langle\theta,\nabla_{\bar\alpha}\nabla_{\beta}\theta\rangle)
\bbr^{-\frac 8 3}\tau^2\\
 \leq &\
C_5\beta_{k-1}\int_{{U_t}(\frac 1
2)}| \theta| ^{2\beta_{k-1}}\bbr^{-4}\tau^2\\
 \leq &\
C_5\beta_{k-1}\int_{{U_t}(\frac 1 4)}| \theta|
^{2\beta_{k-1}}\bbr^{-4}+4^4C_5\beta_{k-1} \int_{ X_t[\frac 1 4]} |
\theta| ^{2\beta_{k-1}}.
\end{aligned}
\end{equation}
From (\ref{0401}), (\ref{0402}) and (\ref{0403}),
\begin{equation*}
-\bigtriangleup_{\dbar} \bbr^{-\frac 8 3}\leq C_6 \bbr^{-4}.
\end{equation*}
Thus  the first term after the inequality in (\ref{090221}) has estimate
\begin{equation}\lab{090227}
\begin{aligned}
&-\int_{{U_t}(\frac 1 2)}| \theta|
^{2\beta_{k-1}}\bigtriangleup_{\bar\partial}(\bbr^{-\frac 8
3}\tau^2)
 \leq C_6\int_{{U_t}(\frac 1
4)}| \theta| ^{2\beta_{k-1}}\bbr^{-4}+4^4C_6\int_{X_ t[\frac 1 4]} |
\theta| ^{2\beta_{k-1}}.
\end{aligned}\end{equation}
Inserting (\ref{090225}) and (\ref{090227}) into (\ref{090221}), we
get
\begin{equation}\lab{090228}
\int_{{U_t}(\frac 1 4)}\bigl|\nabla| \theta|
^{\frac{2\beta_k}{3}}\bigr|^2\bbr^{-\frac 8 3}\tau^2 \leq
C_7\beta_{k-1}\Bigl(\int_{{U_t}(\frac 1 4)}| \theta|
^{2\beta_{k-1}}\bbr^{-4}+
\int_{X_t[\frac 1 4]} | \theta| ^{2\beta_{k-1}}\Bigr);
\end{equation}
inserting (\ref{090228}), (\ref{090210}) and (\ref{090211})
into (\ref{090208}), we obtain
\begin{equation*}\lab{090229}
\Bigl(\int_{{U_t}(\frac 1 4)}| \theta|
^{2\beta_k}\bbr^{-4}\Bigr)^{\frac {1}{\beta_k}} \leq
(C_7\beta_{k-1})^{\frac {1} {\beta_{k-1}}}\Bigl(\int_{{U_t}(\frac 1
4)}| \theta| ^{2\beta_{k-1}} \bbr^{-4}+ \int_{X_t[\frac 1 4]} |
\theta| ^{2\beta_{k-1}}\Bigr)^{\frac{1}{\beta_{k-1}}}.
\end{equation*}

So for any $k\geq 1$, the above inequality implies that either
\begin{equation*}
\Bigl(\int_{{U_t}(\frac 1 4)}| \theta|
^{2\beta_k}\bbr^{-4}\Bigr)^{\frac {1}{\beta_k}} \leq
\bl 2C_7\beta_{k-1}\br^{\frac {1} {\beta_{k-1}}}\Bigl(\int_{{U_t}(\frac 1
4)}| \theta| ^{2\beta_{k-1}} \bbr^{-4}\Bigr)^{\frac{1}{\beta_{k-1}}}
\end{equation*}
or
\begin{equation*}
\Bigl(\int_{{U_t}(\frac 1 4)}| \theta|
^{2\beta_k}\bbr^{-4}\Bigr)^{\frac {1}{\beta_k}} \leq
\bl 2C_7\beta_{k-1}\br^{\frac 1 {\beta_k-1}}\bl\text{vol}(X_t[1/4])\br^{\frac
{1} {\beta_{k-1}}}\sup_{X_t[\frac 1 4]}| \theta| ^2.
\end{equation*}
Since the volume of $X_t[\frac 1 4]$ can be controlled by a constant
independent of $t$, these two inequalities imply that for a constant $C_8$ independent of $t$ and $k$,
\begin{eqnarray*}
\begin{aligned}
\Bigl(\int_{{U_t}(\frac 1 4)}| \theta|
^{2\beta_k}\bbr^{-4}\Bigr)^{\frac {1}{\beta_k}} \leq
&\prod_{i=1}^{k-1}(C_8\beta_{i-1})^{\frac {1}
{\beta_{i-1}}}\Bigl(\int_{{U_t}(\frac 1 4)}| \theta| ^2
\bbr^{-4}+\sup_{X_t[\frac 1 4]}| \theta| ^2\Bigr).
\end{aligned}
\end{eqnarray*}
Taking limit
$k\to \infty$, we get the inequality in the statement of Lemma \ref{313}.
\end{proof}

\begin{proof}[Proof of Lemma \ref{HK}]
We keep the convention introduced in the proof of Lemma \ref{313}.
To streamline the notation, we assign the symbol $\Lambda_t$ to
$$\Lambda_t:= \int_{X_t[\frac 1 8]}(| \gamma| ^2+|
\partial\Phi^{1,3}| ^2).
$$
The Lemma \ref{HK} is then equivalent to that for a constant $C$ independent of $t$,
$$ \int_{{U_t}(\frac 1 4)}| \theta_t| ^2\bbr^{-\frac 8 3}\leq C\Lambda_t.
$$

To begin with, for a smooth positive function $\phi$, we define
$\dbar^\ast_\phi\zeta=\dbar^\ast\zeta-\ast(\partial\log\phi\wedge\ast\zeta)$,
$\nabla_\alpha^\phi=\nabla_\alpha+\partial_\alpha\log\phi$ and $X^{\
\bar\beta}_{\phi\
\bar\beta_k}=-g^{\bar\beta_k\alpha}\partial_{\bar\beta}\partial_{\alpha}\log\phi$.
We need another Kodaira's Bochner formula (\cite{MK}, p.124): For
any $(p,q)$-form $\zeta=\frac{1}{p!q!}\sum \zeta_{\alpha_1\cdots
\bar\beta_q}dz_{\alpha_1}\wedge\cdots\wedge d\bar z_{\beta_q}$,
\begin{eqnarray}\lab{090516}
\begin{aligned}
((\bar\partial\dbar^\ast_\phi+\dbar^\ast_\phi\dbar)\zeta)_{\alpha_1\cdots
\bar\beta_q}=&-\sum_{\alpha,\beta}
g^{\bar\beta\alpha}\nabla_\alpha^\phi\nabla_{\bar\beta}\zeta_{\alpha_1\cdots
\bar\beta_q}\\
+\sum_{i=1}^p&\sum_{k=1}^q\sum_{\alpha,\beta}R^{\alpha\ \ \ \ \
\bar\beta}_{\ \ \alpha_i\bar\beta_k}\zeta_{\alpha_1\cdots
\alpha_{i-1}\alpha\alpha_{i+1}\cdots
\bar\beta_{k-1}\bar\beta\bar\beta_{k+1}\cdots \bar\beta_q}\\
+\sum_{k=1}^q&\sum_{\beta}(X^{\ \bar\beta}_{\phi\
\bar\beta_k}-R_{\bar\beta_k}^{\ \ \bar\beta})\zeta_{\alpha_1\cdots
\bar\beta_{k-1}\bar\beta\bar\beta_{k+1}\cdots \bar\beta_q}.
\end{aligned}
\end{eqnarray}

We let
$\psi=\partial\partial^\ast\gamma$. Over ${U_t}(\ha)$, since the CO-metric is K\"ahler,
we have $\theta=\partial\bar\partial\sta\partial\sta \gamma=-\dbar^\ast\psi$.
We apply the Kodaira formula for $\phi=\phi_1=\bbr^{-\frac 8 3}$
and $\zeta=\psi$.
Since $\psi$ is a $(2,3)$-form and the CO-metric is Ricci flat,
\begin{equation}\lab{0001}
\int_{{U_t}(\frac 1 2)}\langle
\dbar\dbar^\ast_{\phi_1}\psi,\psi\rangle\phi_1\tau =-\int_{V_
t(\frac 1 2)}\langle
g^{\bar\beta\alpha}\nabla_\alpha^{\phi_1}\nabla_{\bar\beta}\psi,\psi\rangle\phi_1\tau
+\int_{{U_t}(\frac 1 2)}\sum_{\beta=1}^3 X^{\ \bar\beta}_{\phi_1\
\bar\beta}| \psi| ^2\phi_1\tau.
\end{equation}
Since $\tau$ has a compact support in ${U_t}(\frac 1 2)$, we compute
\begin{eqnarray*}\lab{090603}
\begin{aligned}
\int_{{U_t}(\frac 1
2)}\langle\dbar\dbar^\ast_{\phi_1}\psi,\psi\rangle\phi_1\tau
=&\int_{{U_t}(\frac 1
2)}\langle\dbar^\ast_{\phi_1}\psi,\dbar^\ast_{\phi_1}\psi\rangle\phi_1\tau+\int_{V_
t(\frac 1
2)}\langle\dbar\tau\wedge \dbar^\ast_{\phi_1}\psi,\psi\rangle\phi_1\\
\leq &\int_{{U_t}(\frac 1 2)}| \theta|
^2\phi_1\tau+\int_{{U_t}(\frac 1
2)}| \partial\log\phi_1\wedge\ast\psi| ^2\phi_1\tau\\
&-2\text{Re}\int_{{U_t}(\frac 1 2)}\langle
\ast(\partial\log\phi_1\wedge\ast\psi),\dbar^\ast\psi\rangle
\phi_1\tau+\cdots,
\end{aligned}
\end{eqnarray*}
where the dots denote terms that are integrations over $X_t[\frac 1
4]$ of smooth function including the derivatives of $\tau$. By
(\ref{307}), the dotted terms are bounded by a fixed multiple, independent of $t$, of
$\Lambda_t=\int_{X_t[\frac 1 8]}(| \gamma| ^2+|
\partial\Phi^{1,3}| ^2)$. In the remainder of this section, the term $C\Lambda_t$
will appear in various places for the same reason.

On the other hand, since $\dbar\dbar^\ast\psi=-\partial\Phi^{1,3}=0$
on ${U_t}(\frac 1 2)$,
\begin{eqnarray*}\lab{090602}
\begin{aligned}
\int_{{U_t}(\frac 1 2)}| \theta| ^2\phi_1\tau
=&\ \text{Re}\int_{{U_t}(\frac 1
2)}\langle\dbar^\ast_{\phi_1}\psi,\dbar^\ast\psi \rangle
\phi_1\tau+\text{Re}\int_{{U_t}(\frac 1 2)}\langle
\ast(\partial\log\phi_1\wedge\ast\psi),\dbar^\ast\psi\rangle
\phi_1\tau\\
\leq &\ \text{Re}\int_{{U_t}(\frac 1 2)}\langle
\ast(\partial\log\phi_1\wedge\ast\psi),\dbar^\ast\psi\rangle
\phi_1\tau+C_1\Lambda_t.
\end{aligned}
\end{eqnarray*}
We remark that the $C_1$ and the $C_i$ to appear later are all independent of $t$.
Combining the above two
inequalities, we get
\begin{eqnarray*}\lab{010901}
\begin{aligned}
\int_{{U_t}(\frac 1
2)}\langle\dbar\dbar^\ast_{\phi_1}\psi,\psi\rangle\phi_1\tau \leq
-\int_{{U_t}(\frac 1 2)}|  \theta| ^2\phi_1\tau+\int_{{U_t}(\frac
1 2)}|
\partial\log\phi_1\wedge\ast\psi| ^2\phi_1\tau+C_2\Lambda_t.
\end{aligned}
\end{eqnarray*}
Inserting the above inequality into (\ref{0001}) and applying
divergence theorem to the first term on the right hand side
(\ref{0001}), since $\psi$ is a $(2,3)$-form, we get
\begin{eqnarray}\lab{010902}
\begin{aligned}
\int_{{U_t}(\frac 1 2)}|  \theta| ^2\phi_1\tau\leq &
\int_{{U_t}(\frac 1 2)}\bigl(|
\partial\log\phi_1| ^2-\sum_{\beta=1}^3 X^{\ \bar\beta}_{\phi_1\
\bar\beta}\bigr)| \psi| ^2\phi_1\tau+C_3\Lambda_t.
\end{aligned}
\end{eqnarray}
According to  (\ref{0401})-(\ref{090705}), we have
\begin{equation*}
| \partial\log\phi_1| ^2\leq \frac 8 3 \bbr^{-\frac 4 3}\and
\sum_{\beta=1}^3 X^{\ \bar\beta}_{\phi_1\ \bar\beta}\geq \frac 8 3
\bbr^{-\frac 4 3}.
\end{equation*}
So from (\ref{010902}), we get
\begin{equation*}
\int_{{U_t}(\frac 1 4)}| \theta| ^2\bbr^{-\frac 4 3}\leq C_3\Lambda_t=
C_3\int_{X_t[\frac 1 8]}(| \gamma| ^2+|
\partial\Phi^{1,3}| )^2.
\end{equation*}
This proves Lemma \ref{HK}.
\end{proof}

\begin{proof}[Proof of Lemma \ref{est}]


For any $0<\iota<\frac 1 3$ and any $0<\delta<\frac 1 4$, by the
H\"older inequality,
\begin{equation*}
\int_{{U_t}(\delta)}| \gamma| ^2\bbr^{-\frac 4 3}\leq
\Bigl(\int_{{U_t}(\frac 1 4)}| \gamma| ^3\bbr^{- 3 \iota}\Bigr)^{\frac
2 3}\Bigl(\int_{U_ t(\delta)}\bbr^{-4+6\iota}\Bigr)^{\frac 1 3}.
\end{equation*}
 Clearly,
\begin{equation*}
\Bigl(\int_{{U_t}(\delta)}\bbr^{-4+6\iota}\Bigr)^{\frac 1 3}=
\Bigl(\frac 2 3 \int_{U_
t(\delta)}\bbr^{-6+6\iota}\text{vol}_{\omega_e}\Bigr)^{\frac 1 3}
\leq C\delta^{2\iota},
\end{equation*}
where the constant $C$ is independent on $t$ and $\delta$.
So to prove the Lemma we only need to prove that for a constant $C$ independent of
$t$,
\begin{equation}\lab{090504}
\Bigl(\int_{{U_t}(\frac 1 4)}| \gamma| ^3\bbr^{- 3\iota}\Bigr)^{\frac
2 3}\leq C\int_{X_t[\frac 1 8]}(| \gamma| ^2+| \partial\Phi^{1,3}|
^2).
\end{equation}

We will prove the above inequality in three steps.
Our first step is to establish the inequality
\begin{equation}\lab{011003}
\Bigl(\int_{{U_t}(\frac 1 4)}| \gamma|
^3\bbr^{-3\iota}\Bigr)^{\frac 2 3} \leq 8\int_{{U_t}(\frac 1 2)}|
\dbar^\ast\gamma| ^2\bbr^{-2\iota}\tau^2+C_1\int_{{U_t}(\frac 1
2)}| \gamma| ^2\bbr^{-2\iota-\frac 4 3}\tau^2+C_1\Lambda_t.
\end{equation}

We now prove this inequality. Using the method in deriving
(\ref{090208}) and (\ref{090221}) for $k=1$, we get
\begin{equation}\lab{090505}
\begin{aligned}
\Bigl(\int_{{U_t}(\frac 1 4)}| \gamma| ^3\bbr^{- 3\iota}\Bigr)^{\frac
2 3} &\leq  C_2\int_{{U_t}(\frac 1 2 )}\bigl|\nabla| \gamma|
\bigr|^2\bbr^{-2\iota}\tau^2+\\
& +C_2\int_{{U_t}(\frac 1 2)}| \gamma| ^2|  \nabla
\bbr^{-\iota-\frac 2 3 }| ^2 \bbr^{\frac 4 3}\tau^2 +
C_2\int_{{U_t}(\frac 1 2)}|  \gamma| ^2\bbr^{-2\iota}|
\nabla\tau| ^2\\
&\leq C_3\int_{{U_t}(\frac 1 2 )}g^{\bar\beta\alpha}(\langle
\nabla_\alpha\gamma,\nabla_\beta\gamma\rangle+\langle\nabla_{\bar\beta}\gamma,
\nabla_{\bar\alpha}\gamma\rangle)\bbr^{-2\iota}\tau^2\\
&\quad+C_3\int_{{U_t}(\frac 1 2)}| \gamma| ^2\bbr^{-2\iota-\frac 4
3}\tau^2 +C_3\Lambda_t.
\end{aligned}
\end{equation}

Let $\phi_2=\bbr^{-2\iota}$ and $\phi_3=\bbr^{-2\iota-\frac 4  3}$.
By divergence theorem,
\begin{equation}\lab{090508}
\begin{aligned}
&\int_{{U_t}(\frac 1
2)}g^{\bar\beta\alpha}\bigl(\langle\nabla_\alpha\gamma,
\nabla_\beta\gamma\rangle+\langle\nabla_{\bar\beta}\gamma,\nabla_{\bar\alpha}\gamma\rangle\bigr)\phi_2\tau^2\\
\leq &-2\int_{{U_t}(\frac 1
2)}g^{\bar\beta\alpha}\langle\nabla_\alpha^{\phi_2}\nabla_{\bar\beta}
\gamma,\gamma\rangle\phi_2\tau^2+\int_{{U_t}(\frac 1 2)}\langle
g^{\bar\beta\alpha}[\nabla_{\bar\beta},
\nabla_{\alpha}]\gamma,\gamma\rangle\phi_2\tau^2+C_4\Lambda\\
&+\int_{{U_t}(\frac 1 2)}g^{\bar\beta\alpha}\langle\partial_\alpha
\log\phi_2\cdot\nabla_{\bar\beta}
\gamma,\gamma\rangle\phi_2\tau^2-\int_{{U_t}(\frac 1
2)}g^{\bar\beta\alpha}\langle\nabla_\alpha\gamma,\partial_\beta(\phi_2\tau^2)\gamma\rangle.
\end{aligned}
\end{equation}
To bound the four terms after the inequality, we use that the
curvature is bounded by $C\bbr^{-\frac 4 3}$ to the second item, and
apply the H\"older inequality to the last two items. After
simplification, we get
\begin{equation}\lab{090515}
\begin{aligned}& \int_{{U_t}(\frac 1 2 )}g^{\bar\beta\alpha}(\langle
\nabla_\alpha\gamma,\nabla_\beta\gamma\rangle+\langle\nabla_{\bar\beta}\gamma,
\nabla_{\bar\alpha}\gamma\rangle)\phi_2\tau^2\\
\leq &-4\int_{{U_t}(\frac 1
2)}g^{\bar\beta\alpha}\langle\nabla_\alpha^{\phi_2}\nabla_{\bar\beta}
\gamma,\gamma\rangle\phi_2\tau^2+C_4\int_{{U_t}(\frac 1 2)}| \gamma|
^2\phi_3\tau^2+C_4\Lambda_t.
\end{aligned}
\end{equation}

We now deal with the first term after ``$\le$'' in the above inequality. We
use  Kodaira's  formula (\ref{090516}) to the case $\zeta=\gamma$ and $\phi=\phi_2$.
Since $\gamma$ is a $(2,3)$-form and CO-metric is Ricci flat, (\ref{090516}) reduces to
\begin{eqnarray*}\lab{090517}
\begin{aligned}
-\sum_{\alpha,\beta}
g^{\bar\beta\alpha}\nabla_\alpha^{\phi_2}\nabla_{\bar\beta}\gamma=
\bar\partial\dbar^\ast_{\phi_2}\gamma-\sum_{\beta=1}^3X^{\ \
\bar\beta}_{\phi_2\ \bar\beta}\gamma.
\end{aligned}
\end{eqnarray*}
So we get
\begin{equation*}\lab{090518}
-\int_{{U_t}(\frac 1 2)}g^{\bar\beta\alpha}\langle
\nabla_\alpha^{\phi_2}\nabla_{\bar\beta}\gamma,\gamma\rangle\phi_2\tau^2
= \int_{{U_t}(\frac 1
2)}\langle\bar\partial\dbar^\ast_{\phi_2}\gamma,\gamma\rangle\phi_2\tau^2-\int_{V_
t(\frac 1 2)}\sum_{\beta=1}^3X^{\ \ \bar\beta}_{\phi_2 \
\bar\beta}| \gamma| ^2\phi_2\tau^2.
\end{equation*}
By the H\"older inequality, we estimate
\begin{equation*}\lab{090519}
\begin{aligned}
\int_{{U_t}(\frac 1 2)}\langle\dbar\dbar^\ast_{\phi_2}\gamma,\gamma \rangle\phi_2\tau^2
=&\int_{{U_t}(\frac 1
2)}\langle\dbar^\ast_{\phi_2}\gamma,\dbar^\ast_{\phi_2}\gamma
\rangle\phi_2\tau^2
-\int_{{U_t}(\frac 1 2)}\langle \dbar\tau^2\wedge\dbar^\ast_{\phi_2}\gamma,\gamma\rangle\phi_2\\
\leq &\,2\!\int_{{U_t}(\frac 1 2)}| \dbar^\ast\gamma| ^2\phi_2\tau^2
+2\int_{{U_t}(\frac 1 2)}|
\partial\log{\phi_2}| ^2| \gamma| ^2\phi_2\tau^2+C_5\Lambda_t.
\end{aligned}
\end{equation*}
Putting together, we get
\begin{equation}\lab{300}
\begin{aligned}
-\int_{{U_t}(\frac 1 2)}g^{\bar\beta\alpha}\langle
\nabla_\alpha^{\phi_2}\nabla_{\bar\beta}\gamma,\gamma\rangle\phi_2\tau^2&
\leq  \int_{{U_t}(\frac 1 2)}\bigl(2|
\partial\log{\phi_2}| ^2-\sum X^{\ \ \bar\beta}_{\phi_2 \
\bar\beta}\bigr)
| \gamma| ^2\phi_2\tau^2+\\
&\quad +2\int_{{U_t}(\frac 1 2)}| \dbar^\ast\gamma|
^2\phi_2\tau^2+C_5\Lambda_t.
\end{aligned}
\end{equation}
On the other hand, by direct calculation,
\begin{equation*}\lab{090520}
| \partial\log{\phi_2} | ^2\leq\frac 3 2 \iota^2  \bbr^{-\frac
4 3}\and
\sum X^{\ \ \bar\beta}_{\phi_2\ \bar\beta}\geq 2\iota \bbr^{-\frac 4
3}.
\end{equation*}
So inequality (\ref{300}) implies
\begin{equation}\lab{011001}
\begin{aligned}
&-\int_{{U_t}(\frac 1 2)}g^{\bar\beta\alpha}\langle
\nabla_\alpha^{\phi_2}\nabla_{\bar\beta}\gamma,\gamma\rangle\phi_2\tau^2\\
\leq &2\int_{{U_t}(\frac 1 2)}| \dbar^\ast\gamma|
^2\phi_2\tau^2+\int_{{U_t}(\frac 1 2)}(3\iota^2-2\iota) | \gamma|
^2\phi_3\tau^2+C_5\Lambda_t.
\end{aligned}
\end{equation}
Finally, inserting (\ref{011001}) into (\ref{090515}) and then inserting (\ref{090515})
into (\ref{090505}), since $0<\iota<\frac 1 3$, we complete our first step in establishing the inequality \eqref{011003}.

Our second step is to prove
\begin{eqnarray}\lab{090528}
\begin{aligned}
\int_{{U_t}(\frac 1 2)}| \gamma| ^2\phi_3\tau^2 \leq
&\frac{2}{2\iota-3\iota^2}\int_{{U_t}(\frac 1 2)}| \dbar^\ast\gamma|
^2\phi_2\tau^2+C_6\Lambda_t.
\end{aligned}
\end{eqnarray}

For this, we first apply the divergence theorem to the left had side of (\ref{011001}):
\begin{eqnarray}\lab{090531}
\begin{aligned}
(2\iota-3\iota^2)\int_{{U_t}(\frac 1 2)}| \gamma| ^2\phi_3\tau^2 \leq&
2\int_{{U_t}(\frac 1 2)}| \dbar^\ast\gamma|
^2\phi_2\tau^2+C_6\Lambda_t.
\end{aligned}
\end{eqnarray}
This inequality implies (\ref{090528}) since when $\iota<\frac 1 3$,
$2\iota-3\iota^2>0$.

Our third step is to prove
\begin{equation}\lab{step3}
\int_{{U_t}(\frac 1 2)}| \dbar^\ast\gamma| ^2\phi_2\tau^2\leq
C_7\int_{X_t[\frac 1 8]}(| \gamma| ^2+\mid
\partial\Phi^{1,3}\mid ^2)).
\end{equation}

To achieve this, we write
\begin{eqnarray}\lab{090500}
\begin{aligned}
2\int_{{U_t}(\frac 1 2)}| \dbar^\ast\gamma| ^2\phi_2\tau^2
=&2\text{Re}\int_{{U_t}(\frac 1
2)}\langle\dbar^\ast_{\phi_2}\gamma,\dbar^\ast\gamma\rangle\phi_2\tau^2\\
+&2\text{Re}\int_{{U_t}(\frac 1 2)}\langle
\ast(\partial\log\phi_2\wedge\ast\gamma),\dbar^\ast\gamma
\rangle\phi_2\tau^2. \end{aligned}
\end{eqnarray}
The first term after the equal sign in (\ref{090500}) is bounded from above by
$$
2\text{Re}\int_{{U_t}(\frac 1 2)}\langle
\gamma,\dbar\dbar^\ast\gamma\rangle\phi_2\tau^2+C_8\Lambda_t,
$$
which is bounded from above by
$$\leq 2b\int_{{U_t}(\frac 1 2)}| \gamma|
^2\phi_3\tau^2+\frac{1}{2b}\int_{{U_t}(\frac 1 2)}|
\dbar\dbar^\ast\gamma| ^2\phi_4\tau^2+C_8\Lambda_t,
$$
for some $b>0$, and for $\phi_4=\bbr^{-2\iota+\frac 4 3}$ and
$\phi_3=\bbr^{-2\iota-\frac 4 3}$. By (\ref{0403}), the second
item after the equal sign in (\ref{090500}) is bounded by
\begin{eqnarray*}
\begin{aligned}
\leq \int_{{U_t}(\frac 1 2)}| \dbar^\ast\gamma| ^2\phi_2\tau^2 +\frac
3 2 \iota^2\int_{{U_t}(\frac 1 2)}| \gamma| ^2\phi_3\tau^2.
\end{aligned}
\end{eqnarray*}
Therefore (\ref{090500}) implies
\begin{equation*}\lab{090534}
\int_{{U_t}(\frac 1 2)}| \dbar^\ast\gamma| ^2\phi_2\tau^2 \leq (\frac
3 2 \iota^2+2b)\int_{{U_t}(\frac 1 2)}| \gamma| ^2\phi_3\tau^2 +\frac
1 {2b} \int_{{U_t}(\frac 1 2)}| \dbar\dbar^\ast\gamma|
^2\phi_4\tau^2+C_8\Lambda_t.
\end{equation*}

Inserting (\ref{090528}) into the above inequality and simplifying, we obtain
\begin{equation}\lab{090536}
\frac{\iota(2-6\iota)-4b}{\iota(2-3\iota)} \int_{{U_t}(\frac 1 2)}|
\dbar^\ast\gamma| ^2\phi_2\tau^2 \leq \frac 1 {2b} \int_{{U_t}(\frac 1
2)}| \dbar\dbar^\ast\gamma| ^2\phi_4\tau^2+ C_9\Lambda_t.
\end{equation}
Since $\iota<\frac 1 3$, for any given $\iota$ we can choose $b$ such
that $\iota(1-3\iota)-2b>0$. Then the above inequality implies
\begin{equation}\lab{090537}
\int_{{U_t}(\frac 1 2)}| \dbar^\ast\gamma| ^2\phi_2\tau^2 \leq
\frac{\iota(2-3\iota)}{2b\bigl(\iota(1-3\iota)-2b\bigr)}
\int_{{U_t}(\frac 1 2)}| \dbar\dbar^\ast\gamma|
^2\phi_4\tau^2+C_{10}\Lambda_t.
\end{equation}

Finally, we need to estimate $\int_{{U_t}(\frac 1 2)}|
\dbar\dbar^\ast\gamma| ^2\phi_4\tau^2$. Since the CO-matric is
K\"ahler and $\partial\gamma=\dbar\gamma=0$, then
$\partial\partial^\ast\gamma=\dbar\dbar^\ast\gamma$. When restricted
to ${U_t}(\frac 1 2)$,
$\dbar\dbar^\ast\partial\partial^\ast\gamma=-\partial\Phi^{1,3}=0$.
From these identities, {since $0<\iota<\frac 1 3$,}
\begin{equation}\lab{090538}
\begin{aligned}
&\int_{{U_t}(\frac 1
2)}| \dbar\dbar^\ast\gamma| ^2r^{-2\iota+\frac 4 3}\tau^2
\leq \int_{{U_t}(\frac 1 2)}| \dbar\dbar^\ast\gamma| ^2\tau^2\\
=&\int_{{U_t}(\frac 1 2)}\langle
\partial\partial^\ast\gamma,\dbar\dbar^\ast\gamma\rangle\tau^2
\leq \int_{{U_t}(\frac 1
2)}\langle\dbar\dbar^\ast\partial\partial^\ast\gamma,\gamma\rangle\tau^2+C_{11}\Lambda_t=C_{11}\Lambda_t.
\end{aligned}
\end{equation}
Combining the above two inequalities, we prove the inequality \eqref{step3}.
This completes our third step.

In the end, we insert (\ref{step3}) into (\ref{090528}), insert
(\ref{090528}) and (\ref{step3}) into (\ref{011003}); we get
(\ref{090504}). This proves Lemma \ref{est}.
\end{proof}

\begin{proof}[Proof of Lemma \ref{est2}] The proof is parallel to that of
the previous Lemma, except that in Lemma \ref{est} the form $\gamma$
is a $(2,3)$-form while in this Lemma $\partial^\ast\gamma$ is a
$(1,3)$-form. Replacing $\gamma$ by $\partial^\ast\gamma$, we find
that all inequalities up to (\ref{090537}) are valid. So to prove
Lemma \ref{est2} we only need to estimate $\int_{{U_t}(\frac 1 2)}|
\dbar\dbar^\ast\partial^\ast\gamma| ^2\bbr^{-2\iota+\frac
43}\tau^2$. Since
$\dbar\dbar^\ast\partial^\ast\gamma=\partial^\ast\partial\partial^\ast\gamma$,
{by the same method in proving (\ref{090538})}, we get
\begin{equation*}
\int_{{U_t}(\frac 1 2)}| \dbar\dbar^\ast\partial^\ast\gamma|
^2\bbr^{-2\iota+\frac 4 3}\tau^2 \leq C_{12}\int_{X_t[\frac 1 8]}(|
\gamma| ^2+| \partial\Phi^{1,3}| ^2).
\end{equation*}
This proves Lemma \ref{est2}.
\end{proof}

\begin{appendix}
\section{The geometry of Candelas-de la Ossa's metrics}

We first recall some notations from Candelas-de la Ossa's paper
\cite{CO}. We consider the family $V_t$:
\begin{equation*}
V_t=\{(w_1,\cdots,w_4)\mid \sum_{i=1}^4(w_i)^2= t\}\sub\CC^4.
\end{equation*}
Since the individual $V_t$ only depend on $|t|$, in the following we shall work with $t>0$.
We let
$r^2=\sum_{i=1}^4|  w_i| ^2$
be the radial coordinate. We set
\begin{equation*} \lab{6004}
\tilde\omega_{\co,t}=i\partial\bar\partial f_t(r^2)
\end{equation*}
The condition that the metric be Ricci-flat is
\begin{equation}\lab{6005}
r^2(r^4- t^2)(\eta_t^3)'+3 t^2\eta_t^3=2r^8\quad \text{with}\quad \eta_t(r^2)=r^2 f'_{ t}(r^2).
\end{equation}
The scale has been chosen so $\eta_t$ has the same asymptotic
behavior as $r^{\frac 4 3}$ for large $r$. After setting
\begin{equation*}
r^2= t\cosh \tau, \ \  {\text{for}\ \  \tau\geq 0}
\end{equation*}
and integrating, we pick the solution
\begin{equation}\lab{6007}
\eta_t=\frac {2^{-1/3} t^{2/3}}{\tanh\tau}(\sinh 2\tau-2\tau)^{1/3}.
\end{equation}
Note that this choice of $\eta_t$ makes the metric regular at $r^2=
t$. {Also note that from (\ref{6005}),
$f'_t(s)=s^{-1}\eta_t(s)$, and that $f_t(s)$
defined in (\ref{pot}) is a solution of this equation.}

In this appendix, we want to estimate the curvature of the CO
metric. Since it is homogeneous (see \cite{CO}), we only need to
perform our calculation at points $q=(\frac{\sqrt{r^2-
t}}{\sqrt{2}},\frac{i\sqrt{r^2- t}}{\sqrt{2}},0, t^{\frac 1 2})$. At
first we pick some orthogonal coordinate at this point. {Since
$dw_1\wedge dw_2\wedge dw_3\ne 0$ near $q$, we can  take
$(w_1,w_2,w_3)$  as a (holomorphic) coordinate in a neighborhood of the
point $q$.} By directly
calculation, we get
$$
\partial\bar\partial r^2|_q=\frac{r^2+ t}{2 t}\bl dw_1\wedge d\bar
w_1+dw_2\wedge d\bar w_2\br+dw_3\wedge d\bar w_3
+i\frac {r^2- t}{2 t}\bl dw_2\wedge d\bar w_1- dw_1\wedge d\bar
w_2\br
$$
and
$$
\partial r^2\wedge \bar\partial r^2|_q =2(r^2- t)dw_2\wedge d\bar
w_2.
$$

To simplify them, we introduce a new coordinate $(u_1,u_2,u_3)$ at the point $q$:
\begin{equation*}
w_1=\frac{2 t}{r_q^2+ t}u_1-i\frac{r_q^2- t}{r_q^2+ t}u_2,\quad
w_2=u_2,\quad w_3=u_3,
\end{equation*}
where $r^2_q\triangleq r^2(q)$. Under this coordinate, the
$\partial\bar\partial r^2$ and $\partial r^2\wedge \bar\partial r^2$
are expressed as
\begin{equation}\lab{0410}
\partial\bar\partial r^2|_q=\frac{2 t}{r^2+ t}du_1\wedge d\bar
u_1+\frac{2r^2}{r^2+ t}du_2\wedge d\bar u_2+du_3\wedge d\bar
u_3
\end{equation}
and
$$
\partial r^2\wedge \bar\partial r^2|_q=2(r^2- t)du_2\wedge d\bar
u_2.
$$
Combined with (\ref{6005}),
\begin{equation}\lab{6019}
f_t'+r^2f_t''=\eta_t'=\frac{2r^8-3 t^2\eta_t^3}{3r^2(r^4-
t^2)\eta_t^2};
\end{equation}
so at the point $q$ the CO  metric is
\begin{equation*}
i\partial\bar\partial f_t|_q = \frac{2 t\eta_t}{r^2(r^2+
t)}idu_1\wedge d\bar u_1+\frac {4r^4}{3\eta_t^2(r^2+ t)}idu_2\wedge
d\bar u_2+\frac {\eta_t}{r^2}idu_3\wedge d\bar u_3.
\end{equation*}

At last we introduce a new coordinate $(z_1,z_2,z_3)$ near the point
$q$ as:
\begin{equation*}
z_1=\Bigl(\frac {2 t\eta_t(q)}{r^2_q(r^2_q+ t)}\Bigr)^{\frac 1
2}u_1,\quad z_2=\Bigl(\frac{4r_q^4}{3\eta_t^2(q)(r^2_q+
t)}\Bigr)^{\frac 1 2}u_2,\quad
z_3=\Bigl(\frac{\eta_t(q)}{r^2_q}\Bigr)^{\frac 1 2}u_3.
\end{equation*}
The CO metric at this point is then expressed as
\begin{equation*}
i\partial\dbar f_t(r^2)|_q=i\sum_{j=1}^3dz_j\wedge d\bar z_j.
\end{equation*}
Under this coordinate, we can rewrite (\ref{0410}) as
$$
\partial\dbar r^2|_q=(r^2)^{\frac 1
3}\Bigl(\frac{r^4}{\eta_t^3}\Bigr)^{\frac 1 3}\Bigr(dz_1\wedge d\bar
z_1+\frac 3 2\frac{\eta_t^3}{r^4}dz_2\wedge d\bar z_2+dz_3\wedge d\bar z_3\Bigr)
$$
and
$$
\partial r^2\wedge \dbar r^2|_q=\frac 3 2(r^2)^{\frac 4
3}\Bigl(\frac{\eta_t^3}{r^4}\Bigr)^{\frac 2 3}\Bigl(1-\frac
{t^2}{r^4}\Bigr)dz_2\wedge d\bar z_2.
$$

To estimate the curvature of the CO metric, we need to investigate the asymptotic behavior of
$\frac{\eta_t^3}{r^4}$.

\begin{lemm} Over $[t,+\infty)$, the function
$r^{-4}{\eta_t^3}$ is an increasing function and
\begin{equation*} \lab{6030}
\lim_{r^2\to  t}r^{-4}{\eta_t^3}=\frac 2 3,\quad \lim_{r^2\to
\infty}r^{-4}{\eta_t^3}=1.
\end{equation*}
\end{lemm}

\begin{proof}
 Let $h(\tau)=r^{-4}{\eta_t^3}$. From (\ref{6007}),
\begin{eqnarray*}
h(\tau)=\frac 1 2 \frac{\cosh\tau(\sinh2\tau-2\tau)}{\sinh^3\tau}.
\end{eqnarray*}
Differentiating,
\begin{equation*}
h'(\tau)=\frac 1 {2\sinh ^4\tau}h_1(\tau)\quad\text{where}\quad
h_1(\tau)=4\tau+e^{2\tau}(\tau- 3/ 2)+e^{-2\tau}(\tau+3/2)
\end{equation*}
and
\begin{equation*}
h_1'(\tau)=2\tau e^{2\tau}-2e^{2\tau}-2\tau
e^{-2\tau}-2e^{-2\tau}+4.
\end{equation*}
Thus for any $\tau>0$,
\begin{equation*}
h_1'(\tau)=4\tau\sum_{n=1}^{\infty}\frac{(2\tau)^{2n+1}}{(2n+1)!}\cdot\frac
n {n+1}>0.
\end{equation*}
Since $h_1(0)=0$, $h_1(\tau)>0$ and $h'(\tau)>0$. So over
$[0,+\infty)$, the function $r^{-4}{\eta_t^3}$ is a increasing
function of $\tau$. Since $r^2=t\cosh\tau$ for $\tau\geq 0$ is an
increasing function in $\tau$, $r^{-4}\eta^3$ is increasing in
$r^2$. Since  $\tau\to 0$ when $r^2\to t$, and $\tau\to \infty$ when
$r^2\to\infty$, we obtain the two desired limits by applying the
L'Hospital rule. This proves the Lemma.
\end{proof}

We next investigate $\eta_t'$. From $\eta_t^3=r^4 h(\tau)$,
\begin{equation*}
3\eta_t^2\eta_t'=2r^2h(\tau)+r^4h'(\tau)\frac{d\tau}{dr^2}=2r^2h(\tau)+
t^{-1}r^4h'(\tau)\sinh^{-1}\tau>0,
\end{equation*}
hence $r^{\frac 2 3}\eta_t'>0$. On the other hand by (\ref{6019}),
we get
\begin{equation*}
r^{\frac 2 3}\eta_t'=\frac{2-3\frac{
t^2}{r^4}\frac{\eta_t^3}{r^4}}{3(1-\frac{
t^2}{r^4})(\frac{\eta_t^3}{r^4})^{\frac 2
3}}=\bigl(\frac{\eta_t^3}{r^4}\bigr)^{\frac 1
3}+\frac{2-3\frac{\eta_t^3}{r^4}}{3(1-\frac
{t^2}{r^4})(\frac{\eta_t^3}{r^4})^{\frac 2 3}}.
\end{equation*}
Then from Lemma \ref{6030}, we see that
$$0<r^{\frac 2 3}\eta_t'<1.
$$

In the following for two functions $\alpha(r,t)$ and $\beta(r,t)$ in $r$ and $t$
we shall use $\alpha\lesssim \beta$  to mean that
there is a
constant $C$ independent on $r$ and $ t$ such that
\begin{equation*}
|  \alpha(r,t)| \leq C|  \beta(r,t)| .
\end{equation*}
Under this convention, the previous Lemma and the last inequality can be abbreviated as
\begin{equation*}
\eta_t\lesssim r^{\frac 4 3} \and  \eta_t'\lesssim  r^{-\frac 2 3}.
\end{equation*}
For higher derivatives, by introducing $\epsilon=\frac { t}{r^2}$, the identities (\ref{6005}) and
(\ref{6019}) imply
\begin{equation*}
\eta_t''\lesssim  r^{-\frac 8 3}(1-\epsilon)^{-1}, \quad
\eta_t^{(3)}\lesssim  r^{-\frac {14} 3}(1-\epsilon)^{-2}.
\end{equation*}
Therefore, by the second identity of (\ref{6005}), we obtain the following asymptotic estimates
\begin{equation*}
f'_t\lesssim  r^{-\frac 2 3}, \quad f_t ''\lesssim  r^{-\frac 8 3}, \quad f_{
t}^{(3)}\lesssim  r^{-\frac {14} 3}(1-\epsilon)^{-1} \and f_{
t}^{(4)}\lesssim  r^{-\frac {20}{3}}(1-\epsilon)^{-2}.
\end{equation*}

To proceed, we need to the partial derivatives of $r^2$ with respect to $z_i$ and $\bar z_i$. For
simplicity, we shall use the subscript $i$ to denote the partial derivative with respect to $z_i$,
and use $\bar i$ for derivatives with respect to $\bar z_i$. Thus, for instance, $\frac{\partial^2
r^2}{\partial z_i\bar\partial z_j}=(r^2)_{i\bar j}$.

Under this convention, we compute directly that at
the point $q$, the first order partial derivatives
\begin{equation}\lab{0529}
(r^2)_1=(r^2)_3=0\and (r^2)_2=-\frac{\sqrt{6}}{2}i(r^2- t)^{\frac 1
2}(r^2+ t)^{\frac 1 2}\frac{\eta_t}{r^2}\lesssim r^{\frac 4
3}(1-\epsilon)^{\frac 1 2};
\end{equation}
the second order derivatives $(r^2)_{i\bar j}=0$ except the following
\begin{equation}\lab{0530}
(r^2)_{1\bar 1}=(r^2)_{3\bar 3}
=\frac{r^2}{\eta_t}\lesssim  r^{\frac 2 3},\quad
(r^2)_{2\bar 2}=\frac 3 2
\frac{\eta_t^2}{r^2}\lesssim  r^{\frac 2 3};
\end{equation}
the second derivatives $(r^2)_{i j}=0$ except the following
\begin{equation*}\lab{0531}
(r^2)_{11}=(r^2)_{33}=-\frac{r^2}{\eta_t}\lesssim  r^{\frac 2
3},\quad (r^2)_{22}=-\frac 3 2 \frac{\eta_t^2}{r^2}\epsilon \lesssim
r^{\frac 2 3}.
\end{equation*}
For the third order partial derivatives of type $i\bar j k$, we have
the vanishing $(r^2)_{i\bar j k}=0$ except the following
\begin{eqnarray*}
\begin{aligned}
&(r^2)_{1\bar 1 1}=\frac{r^{3}(1-\epsilon)^{\frac 1 2}}{ t^{\frac 1
2}\eta_t^{\frac 3 2}(1+\epsilon)^{\frac 1 2}} \lesssim
\epsilon^{-\frac 1 2}(1-\epsilon)^{\frac 1 2},\quad (r^2)_{1\bar 2
1}=-i\frac{\sqrt{6}}{2}\frac{(1-\epsilon)^{\frac 1
2}}{(1+\epsilon)^{\frac 1 2}} \lesssim (1-\epsilon)^{\frac 1
2},\\
&(r^2)_{2\bar 1 2}=\frac{3}{2}\frac{ t^{\frac 1 2}\eta_t^{\frac 3
2}(1-\epsilon)^{\frac 1 2}}{r^3(1+\epsilon)^{\frac 1 2}} \lesssim
\epsilon^{\frac 1 2}(1-\epsilon)^{\frac 1 2},\quad\ \ \ (r^2)_{2\bar
2 2}=-i\frac{3\sqrt{6}}{4}\frac{ t\eta_t^3(1-\epsilon)^{\frac 1
2}}{r^6(1+\epsilon)^{\frac 1 2}} \lesssim   \epsilon
(1-\epsilon)^{\frac 1 2},
\\
& (r^2)_{3\bar 1 3}=\frac{r^3(1-\epsilon)^{\frac 1 2}}{ t^{\frac 1
2}\eta_t^{\frac 3 2}(1+\epsilon)^{\frac 1 2}}\lesssim
\epsilon^{-\frac 1 2}(1-\epsilon)^{\frac 1 2},\quad \ \ \
(r^2)_{3\bar 2 3}=-i\frac{\sqrt{6}}{2}\frac{(1-\epsilon)^{\frac 1
2}}{(1+\epsilon)^{\frac 1 2}}\lesssim  (1-\epsilon)^{\frac 1 2}.
\end{aligned}
\end{eqnarray*}
For the fourth order partial derivatives, we still have the
vanishing except the following
\begin{eqnarray*}
\begin{aligned}
&(r^2)_{1\bar 1 1 \bar 1}=\frac{r^4}{ t\eta_t^2}\lesssim  r^{\frac 4
3} t^{-1},\ \ \ (r^2)_{1\bar 2 1\bar 2}=(r^2)_{2\bar 12\bar
1}=\frac{3\eta_t}{2r^2}\lesssim r^{-\frac 2
3},\\
&(r^2)_{2\bar 2 2\bar 2}=\frac{9 t\eta_t^4}{4r^8}\lesssim tr^{-\frac
8 3},\ \ \ (r^2)_{1\bar 3 1\bar 3}=(r^2)_{3\bar 13\bar
1}=\frac{r^4}{ t\eta_t^2}\lesssim r^{\frac 4 3}
t^{-1},\\
&(r^2)_{3\bar 33\bar 3} =\frac{r^4}{ t\eta_t^2}\lesssim  r^{\frac 4
3} t^{-1},\ \ \  (r^2)_{2\bar 32\bar 3}=(r^2)_{3\bar 23\bar
2}=\frac{3\eta_t}{2r^2}\lesssim r^{-\frac 2 3}.
\end{aligned}
\end{eqnarray*}

We now use these asymptotic estimate of the partial derivatives of $r^2$
to prove  Lemma \ref{cur}.  {Since $(z_1,z_2,z_3)$ is the orthogonal coordinate at the point $p$, we only need to prove
that {\sl there is a constant $C$ independent of $t$ and $r$ such that the
curvature tensor $R_{i\bar j k\bar l}$ of the CO metric}
$\tilde\omega_{\co,t}$ at $q$ has bound
\begin{equation}\lab{bound-1}
R_{i\bar jk\bar l}\lesssim r^{-\frac  4 3}.
\end{equation}

The curvature at $q$ has the form}
\begin{equation*}
R_{i\bar j k\bar l}=-(f_t)_{i\bar jk\bar l}+(f_t)_{ik\bar
q}(f_t)_{q\bar j\bar l}.
\end{equation*}
One group of terms appearing in $(f_t)_{i\bar jk\bar l}$ are of the
type
$$
f_t^{(4)}\cdot(r^2)_i (r^2)_{\bar j} (r^2)_{k} (r^2)_{\bar l},\quad
f_t^{(3)}\cdot
\sum
(r^2)_{i_1 i_2} (r^2)_{i_3} (r^2)_{ i_4},\quad f_t''\cdot
\sum(r^2)_{i_1 i_2} (r^2)_{i_3 i_4},
$$
where the summations are taken over all possible permutation
$\{i_1,i_2,i_3,i_4\}=\{i,\bar j, k,\bar l\}$. For such type of
terms, using the previous estimate, we check directly that they are
bounded by $C r^{-\frac{4}{3}}$.

The other group of terms in appearing $(f_t)_{i\bar j k\bar l}$ are
of the type:
\begin{align}
&f_t'\cdot (r^2)_{i\bar jk\bar l}\\
&f_t''\cdot \bigl((r^2)_{i\bar j k} (r^2)_{\bar l}+(r^2)_{i\bar l
k}(r^2)_{\bar j}\bigr),\\
&f_t''\cdot\bigl((r^2)_{\bar j k\bar l}(r^2)_{i}+(r^2)_{\bar j i\bar
l}(r^2)_{k}\bigr).
\end{align}
Of these, (A.8) vanishes when $i\not=k$ or $j\not=l$, (A.9) vanishes
when $i\not=k$ and (A.10) vanishes when $j\not=l$. The remaining
cases in (A.8)-(A.10) may be not vanishing, and will be treated
separatly momentarily.

We now look at the product term $(f_t)_{ik\bar q}(f_t)_{q\bar j\bar
l}$. First in the expression of $(f_ t)_{ik\bar q}$, the following
two types of terms
\begin{equation*}
f_t^{(3)}\cdot (r^2)_i(r^2)_k (r^2)_{\bar q} \and f_t''\cdot
\sum_{i_1,i_2,i_3}(r^2)_{i_1i_2}(r^2)_{i_3}
\end{equation*}
are  bounded by $C r^{-\frac{2}{3}}$; therefore corresponding
product terms
\begin{equation*}
\begin{aligned}
&\Bigl(f_t^{(3)}\cdot (r^2)_i(r^2)_k (r^2)_{\bar q}+f_t''\cdot
\sum_{i_1,i_2,i_3}(r^2)_{i_1i_2}(r^2)_{i_3}\Bigr)\\
\times& \Bigl(f_t^{(3)}\cdot (r^2)_{\bar j}(r^2)_{\bar l} (r^2)_{
q}+f_t''\cdot \sum_{j_1,j_2,j_3}(r^2)_{j_1j_2}(r^2)_{j_3}\Bigr)
\end{aligned}
\end{equation*}
in the expansion of $(f_t)_{ik\bar q}(f_t)_{q\bar j\bar l}$ are also
bounded by $Cr^{-\frac 4 3}$. Here the summations are over all
possible permutation $\{i_1,i_2,i_3\}=\{i,k,\bar q\}$ and
$\{j_1,j_2,j_3\}=\{\bar j,\bar l, q\}$.

The remainding  terms in $(f_t)_{ik\bar q}(f_t)_{q\bar j\bar l}$ are
of the following types:
\begin{align}
 & (f_t')^2\cdot (r^2)_{ik\bar q}(r^2)_{\bar j\bar l q},\\
&f_t'\cdot(r^2)_{ik\bar q}\cdot
\bigl(f_t''\sum(r^2)_{i_1i_2}(r^2)_{i_3}+f_ t^{(3)}\cdot (r^2)_{\bar
j}(r^2)_{\bar l}(r^2)_q\bigr),\\
&f_t'\cdot (r^2)_{\bar j\bar l q}\cdot \bigl(f_t''\sum
(r^2)_{i_1i_2}( r^2)_{i_3}+f_ t^{(3)}\cdot
(r^2)_{i}(r^2)_{k}(r^2)_{\bar q}\bigr),
\end{align}
where the summation in the second line is taken over all possible
permutation $\{i_1,i_2,i_3\}=\{\bar j,\bar l,q\}$ and summation in
the last line is taken over all possible permutation
$\{i_1,i_2,i_3\}=\{i,k,\bar q\}$. Like before, they vanish when
$i\not=k$ in case (A.11) and (A.12) or when $j\not=l$ in case (A.11)
and (A.13).

Combining the above discussion, we see that when $i\not= k$ and
$j\not= l$, the bound \eqref{bound-1} follow immediately. For
others, we need to treat case by case.

For the case $i\not= k$ and $j=l$, we have
\begin{equation*}
R_{i\bar jk\bar j}\lesssim  r^{-\frac 4 3}-f''_ t\cdot (r^2)_{\bar
j\bar
j i}(r^2)_k\bigl(1-f'_t\cdot (r^2)_{i\bar i}\bigr) -f_t''\cdot (r^2)_{\bar j\bar j k} (r^2)_{i}\bigl(1-f_
t'\cdot (r^2)_{k\bar k}\bigr).
\end{equation*}
We claim that the last term is always zero. Indeed, because of
(\ref{0529}), if $i\not= 2$, then the last term is equal to zero. If
$i=2$, then $k\not=2$, and by (\ref{6005}) and (\ref{0530}) we have
$1-f_t'\cdot (r^2)_{k\bar k}=0$. This proves the claim that the last
item always vanishes. For the same reason, the second item vanishes.
Therefore, when $i\ne k$ and $j=l$, the estimate \eqref{bound-1}
holds.

For the case $i=k$ and $j\not= l$, we have $R_{i\bar ji\bar
l}=\overline{R_{j\bar i l\bar i}}\lesssim  r^{-\frac 4 3}$. This
proves the bound \eqref{bound-1} in this case.

Finally, we need to consider the cases $i=k$, $j=l$ and $i\not =j$. We should consider
these cases individually. In case $i\bar j k\bar l=1\bar 3 1\bar 3$, we have
\begin{equation*}
\begin{aligned}
R_{1\bar 31\bar 3}\lesssim  &r^{-\frac 4 3}-f'_{ t}\cdot
(r^2)_{1\bar 31 \bar 3}+(f'_t)^2\cdot (r^2)_{11\bar 1}(r^2)_{\bar
3\bar 31}\\
& +\bigl(f_ t'\cdot (r^2)_{11\bar 2}+f_ t''\cdot
(r^2)_{11}(r^2)_{\bar 2}\bigr)\cdot \bigl(f_ t'\cdot (r^2)_{\bar
3\bar 3 2}+f_ t''\cdot (r^2)_{\bar 3\bar 3}(r^2)_{2}\bigr).\\
\end{aligned}
\end{equation*}
The last term is clearly $\lesssim $ $r^{-\frac 4 3}$; the second and  third items combined give
\begin{equation*}
-f'_{ t}\cdot (r^2)_{1\bar 31 \bar 3}+(f'_t)^2\cdot (r^2)_{11\bar
1}(r^2)_{\bar 3\bar 31}=-f_t'\frac{r^4}{ t\eta_t^2}+(f'_t)^2
\frac{r^6(r^2- t)}{ t\eta_t^3(r^2+ t)} =\frac{-2r^2}{\eta_t(r^2+
t)}\lesssim  r^{-\frac 4 3}.
\end{equation*}
This proves the bound \eqref{bound-1} for $
R_{1\bar 31\bar 3}$.
By similar method, we obtain desired bound \eqref{bound-1} for
$R_{1\bar 2 1\bar 2}$ and $R_{2\bar 3 2\bar 3}$.

Finally, we consider the case $i=j=k=l$. Since the metric is Ricci-flat, we have
\begin{equation*}
R_{i\bar i i\bar i}=-\sum_{j\not =i}R_{i\bar i j\bar j}\lesssim
r^{-\frac 4 3}.
\end{equation*}
This completes the proof of Lemma \ref{cur}.

\begin{rema}
We remark that when $r\to  t$, the induced metric on the surface $r^2=\text{const.}$ approaches
\begin{equation*}
\frac  1 2 \Bigl(\frac{2 t^2}{3}\Bigr)^{\frac 1 3}ds^2|_{S^3},
\end{equation*}
where $ds^2|_{S^3}$ is the standard metric on $S^3$. The curvature
of the limiting metric is $C t^{-\frac 2 3}$ for some constant $C$.
\end{rema}
\end{appendix}

\end{document}